\documentclass{book}
\usepackage[contrib,lang=british]{ems-book} 

\newcommand{\Pp}{\mathbb{P}}
\newcommand{\Z}{\mathbb{Z}}
\newcommand{\N}{\mathbb{N}}
\newcommand{\R}{\mathbb{R}}

\newcommand{\Ccal}{\mathcal{C}}
\newcommand{\Dcal}{\mathcal{D}}
\newcommand{\Fcal}{\mathcal{F}}
\newcommand{\Tcal}{\mathcal{T}}
\newcommand{\Xcal}{\mathcal{X}}

\newcommand{\Cfrak}{\mathfrak{C}}
\newcommand{\Rfrak}{\mathfrak{R}}

\DeclareMathOperator{\bdim}{\boldsymbol{\dim}}
\newcommand{\bc}{\boldsymbol{c}}

\DeclareMathOperator{\Hom}{Hom}

\DeclareMathOperator{\mods}{mod}
\DeclareMathOperator{\End}{End}

\DeclareMathOperator{\Ext}{Ext}
\DeclareMathOperator{\image}{im}
\DeclareMathOperator{\coker}{coker}
\DeclareMathOperator{\add}{add}
\DeclareMathOperator{\codim}{codim}
\DeclareMathOperator{\Filt}{Filt}
\DeclareMathOperator{\Fac}{Fac}
\DeclareMathOperator{\rank}{rk}
\DeclareMathOperator{\Sub}{Sub}
\DeclareMathOperator{\pd}{pd}

\usepackage[capitalise]{cleveref}

\theoremstyle{plain}
\newtheorem{thm}{Theorem}
\newtheorem{lem}[thm]{Lemma}
\newtheorem{prop}[thm]{Proposition}

\theoremstyle{definition}
\newtheorem{defn}[thm]{Definition}
\newtheorem{exmp}[thm]{Example}

\usepackage{eucal,mathtools}
\usepackage{listings}
\usepackage{tikz-cd}
\usepackage{tikz}
\usepackage{indentfirst}
\usepackage{microtype}
\usetikzlibrary{arrows}
\usepackage{stmaryrd}
\usepackage{verbatim}
\usepackage{adjustbox}

\begin{document}
\mainmatter

\title{Wall-and-chamber structures for finite-dimensional algebras and $\tau$-tilting theory}
\titlemark{The wall-and-chamber structure of an algebra}

\emsauthor{1}{Maximilian Kaipel}{M.~Kaipel}
\emsauthor{2}{Hipolito Treffinger}{H.~Treffinger}


\emsaffil{1}{Postal address \email{Email}}

\classification[16G10, 16G20, 16G99, 16E20]{16-06}




\begin{abstract}
    The \textit{wall-and-chamber} structure is a geometric invariant that can be associated to any algebra. 
    In this notes we give the definition of this object and we explain its relationship with torsion classes and $\tau$-tilting theory.
\end{abstract}

\makecontribtitle


\section{Introduction}

The main aim of representation theory of finite-dimensional algebras is to understand the category of (finitely presented) modules over a given algebra. 
One of the founding results in this area is a result by Gabriel \cite{Gabriel}, which states that the module category of any finite-dimensional algebra over an algebraically closed field is equivalent, in the sense of Morita \cite{Morita}, to the category of representations of a quiver with relations (which depend on the original algebra).
This was breakthrough in the field, not only because it reduces greatly the universe of algebras to be studied, but also because with the incorporation of quivers, Gabriel also allowed the introduction of important combinatorial tools that have played a central role in the theory ever since.

Several years later, at the turn of the century, Fomin and Zelevinski introduced \textit{cluster algebras} \cite{FominZelevinsky2001} as a new approach to understand Lusztig's dual canonical bases. 
These are algebras that are defined from starting data, known as the \textit{the initial seed}, and then the reminding data is constructed using a combinatorial process known as \textit{mutation}.
For a family of cluster algebras of particular importance, the \textit{antisymmetric} cluster algebras, both the initial data and the mutation process can be encoded in terms of quivers.
As a consequence, many mathematicians started using the tools developed throughout the years in representation theory to understand cluster algebras and solve some of the standing conjectures in this new subject \cite{GeissLeclercSchroer, BuanMarshReinekeReitenTodorov, CalderoChapotonSchiffler, Amiot}.

Also, the study of cluster algebras made explicit certain patterns in the module category of every finite-dimensional algebra. 
It was in this context that notions which are central in current representation theory, such as $\tau$-tilting theory \cite{AIR2014} or higher homological algebra \cite{Iyama2007b} \cite{Iyama2007a}, were defined. 
For a detailed account of the origin of $\tau$-tilting theory, see \cite{Treffinger2021}.

Cluster algebras are subalgebras of certain function fields. 
This point of view allows a geometric approach to their study. 
The basic idea states that a cluster algebra should be the coordinate ring of a variety, which is unknown a priori. 
In this line of work one can find, for instance, the works of Fock and Goncharov \cite{FockGoncharov} and Gross, Hacking, Keel and Kontsevich \cite{GHKK2018}. 

In the latter the authors associate to each cluster algebra a \textit{cluster scattering diagram}, a cone complex in $\mathbb{R}^n$ where every cone comes with the element of a group satisfying equations that are determined by the cone complex itself.
They then used these scattering diagrams to show at once several open conjectures in cluster theory and to give new proofs for other conjectures that have been settled by different means.

Soon after, Bridgeland \cite{Bridgeland} defined for every finite-dimensional algebra a \textit{stability} scattering diagram. 
In this case the cone complex is defined using the stability conditions studied by King in \cite{King1994}.
Bridgeland showed moreover that, if the algebra is hereditary (i.e. $A = KQ$ for an acyclic quiver $Q$) over an algebraically closed field $K$, then the stability scattering diagram associated to $A$ is isomorphic to the cluster scattering diagram associated to the cluster algebra of $Q$.

The support of the stability scattering diagram, that is, its underlying cone complex, is known as the \textit{wall-and-chamber structure} of the algebra. 
These notes\footnote{These are the lecture notes of the course titled \textit{``Wall-and-chamber structures for finite dimensional algebras''} given by the second named author during the Workshop of the 20th edition of the International Conference on Representation of Algebras (ICRA 2022) held in Montevideo, Uruguay, between the 3rd and the 6th of August 2022.} are dedicated to give a detailed definition of the wall-and-chamber structure of an algebra and to explain the relationship of this object and $\tau$-tilting theory \cite{AIR2014}.

The structure of these notes is the following. After recalling some basic facts about the module category of an algebra, we introduce stability conditions, the wall-and-chamber structure of an algebra and we study some of their basic properties.
Afterwards we give a brief overview on $\tau$-tilting theory.
We then explain the relationship between the wall-and-chamber structure of an algebra and its $\tau$-tilting theory.
We finish the article by illustrating most of the results in a particular example. 

\section{Preliminaries} \label{sec:prelim}

For any finite-dimensional algebra $A$ we denote by $\mods A$ the category of finitely presented (right) $A$-modules. 
Let $M \in \mods A$, then we denote by $\add M$ the full subcategory of $\mods A$ additively generated by $M$, that is the category of all direct summands of direct sums of $M$. Furthermore, define the subcategories $\Fac M$ and $\Sub M$ to be given by
\[ \Fac M \coloneqq \{ X \in \mods A: \text{there exists an epimorphism } M^p \to X \to 0 \text{ for some $p \in \N$} \},\]
\[ \Sub M \coloneqq \{ X \in \mods A:  \text{there exists a monomorphism }  0 \to X \to M^p \text{ for some $p \in \N$}\}.\]
Finally, we need to introduce the following notation:
\[ T^\perp \coloneqq \{X \in \mods A: \Hom(T,X)=0 \}, \quad {}^\perp T \coloneqq \{X \in \mods A: \Hom(X,T) = 0\}.\]

In these notes we often work with a very particular type of subcategories of $\mods A$, the so-called \textit{torsion classes} \cite{Dickson66}.

\begin{defn} \label{def:torsionpair}\cite[Theorem 2.1]{Dickson66}
A pair of full subcategories $(\Tcal, \Fcal)$ of $\mods A$ is a \textit{torsion pair} if the following are satisfied:
\begin{enumerate}
    \item $\Hom_A(T,F)=0$ for $T \in \Tcal$ and $F \in \Fcal$, and
    \item for all $M \in \mods A$ there exists a short exact sequence given by
    \begin{equation}\label{eq:canseq} 0 \to tM \to M \to fM \to 0 \end{equation}
    where $tM \in \Tcal$ and $fM \in \Fcal$. 
    Then $\Tcal$ is called the \textit{torsion class} and $\Fcal$ the \textit{torsion-free class}.
\end{enumerate}
\end{defn}
It turns out that given a torsion pair $(\Tcal, \Fcal)$ this short exact sequence \ref{eq:canseq} is unique, see \cite[Proposition VI.1.5]{bluebookV1}, meaning that $tM$ and $fM$ depend functorially on $M$. We then call $0 \to tM \to M \to fM \to 0$ the \textit{canonical} short exact sequence of $M$ with respect to $(\Tcal, \Fcal)$. 
In this case we say that $tM$ is the \textit{torsion object} of $M$.

While torsion pairs are defined for an arbitrary abelian category, the following result gives us an additional characterisation of the torsion and torsion-free classes in our setting.

\begin{prop}\label{prop:torsionIFF} \cite[Theorem 2.3]{Dickson66}
A subcategory $\Tcal$ of $\mods A$ is a torsion class if and only if $\Tcal$ is closed under quotients and extensions.  In this case, the corresponding torsion-free class $\Fcal$ such that $(\Tcal, \Fcal)$ is a torsion pair is $\Fcal = \Tcal^\perp$.  

Dually, a subcateogory $\Fcal$ of $\mods A$ is a torsion-free class if and only if $\Fcal$ is closed under subobjects and extensions. In this case, the associated torsion class is $\Tcal = {}^\perp \Fcal$.
\end{prop}

We furthermore assume that $A$ is a basic algebra. 
This means that there exist a collection of idempotents $\{ e_1, \dots, e_n\}$ such  that $1_A = \sum_{i=1}^n e_i$ verifying $e_iA \not \cong e_jA$ if $i \neq j$.
If we denote $e_iA$ as $P(i)$, we can therefore write $A= \bigoplus_{i=1}^n P(i)$.
Moreover, in this case, the set $\{P(1), \dots, P(n)\}$ contains exactly one representative of each isomorphism class of indecomposable projective $A$-modules.

For a module $M \in \mods A$, let $|M|$ denote the number of isomorphism classes of indecomposable direct summands. 
In particular, with the notation above, we have that $|A|=n$. 

The \textit{Grothendieck group} of $\mods A$ is the abelian group $K_0( \mods A) \coloneqq K_0(A) = \Fcal/\Fcal'$, where $\Fcal$ is the free abelian group generated by the set of isomorphism classes of a module $M \in \mods A$. Then, the Grothendieck group $K_0(A)$ is obtained by identifying $[M] = [L] + [N]$ corresponding to all short exact sequences
\begin{equation} 0 \to L \to M \to N \to 0 \label{eq:ses}\end{equation}
in $\mods A$. 
The Jordan-Hölder Theorem for modules states that every module $M \in \mods A$ has a composition series with unique length whose subfactors are simple modules which are unique up to permutation. 
In other words there exists a composition series
\[ 0 = M_0 \subset M_1 \subset \dots \subset M_t = M \]
which comes with short exact sequences
\[ 0 \to M_{j-1} \to M_j \to \underbrace{M_j /M_{j-1}}_{\cong S(i)} \to 0.\]
We may therefore decompose the equivalence class of any module $M \in K_0(A)$ as follows:
\begin{align*}
 [M]  = [M_t /M_{t-1}] + [M_{t-1}] = \dots  = \sum_{j=1}^t [M_{j} / M_{j-1}] = \sum_{i=1}^n a_i [S(i)],
\end{align*}
where $a_i$ is the multiplicity of the simple module $S(i)$ at vertex $i$ as a composition factor of $M$. 
Then $\{[S(1)], \dots, [S(n)]\}$ generates $K_0(A)$ and there is an isomorphism of abelian groups given by $\bdim: K_0(A) \to \Z^n$ sending $[S(i)] \mapsto \mathbf{e}_i$, where $\{\mathbf{e}_1, \dots, \mathbf{e}_n\}$ is the canonical basis of $\mathbb{Z}^n$.
Given an $A$-module $M$, we denote by $\bdim M$ the vector $\bdim ([M])\in \mathbb{Z}^n$ and we call it the \textit{dimension vector} of $M$. 
The reason behind this name is the following: If $A = KQ/I$ is the bounded path algebra of a quiver $Q=(Q_0, Q_1)$ with set of vertices $Q_0=\{1, \dots, n\}$ and $M$ is an $A$-module, then $\bdim M = (\dim_K M_1, \dots, \dim_K M_n)$ where $M_i$ is the $K$-vector space $M\mathbf{e}_i$.

Let 
\( D_A \coloneqq \text{diag}(\dim_K\End_A(S(1)), \dots , \dim_K\End_A(S(n)))\)
be the matrix whose diagonal entries equal the dimensions over $K$ of the endomorphism algebras of the simples. 
We can now define the inner product we use throughout these lecture notes as $\langle -,-\rangle: \R^n \times \R^n \to \R$ as
\[ \langle v, w \rangle = v^T D_A w. \]
Note that when $A=KQ/I$ the matrix $D_A$ is the identity matrix and this inner product corresponds with the classical dot product on $\R^n$.
Due to this remark, and by abuse of notation, the classical dot product in $\mathbb{R}^n$ is also denoted by $\langle - , - \rangle$.

\section{Stability conditions}

The study of stability conditions of module categories is due to King \cite{King1994}, who translated the geometric invariant theory of Mumford \cite{Mumford65} to quiver representations. We begin by recalling some basic concepts and definitions. Given a vector in $\R^n$, a so called \textit{stability condition}, we get two notions of stability for modules.

\begin{defn}\cite{King1994} Let $M \in \mods A$, take $v \in \Z^n \otimes \R = \R^n$. We say $M$ is \textit{$v$-semistable} if $\langle v, \bdim M \rangle = 0$ and if for any nonzero proper subobject $L$ of $M$ we have $\langle v, \bdim L \rangle \leq 0$ or, equivalently, we have $\langle v , \bdim N \rangle \geq 0$ for all nonzero proper quotients $N$ of $M$. We say $M$ is \textit{$v$-stable} if these inequalities are strict.
\end{defn}

\begin{defn} Define the \textit{category of $v$-semistable objects} $\mods_v^{ss}A$ to be the full subcategory of $\mods A$ whose objects are $v$-semistable.
\end{defn}

The following result is a special case of \cite[Proposition 2.20]{BSTStability} and shows that $\mods_v^{ss} A$ has desirable properties.

\begin{thm} \cite[Proposition 2.20]{BSTStability} The category $\mods_v^{ss}A$ is a wide subcategory of $\mods A$ i.e. closed under kernels, cokernels and extensions.
\end{thm}

\begin{proof}
Let $v \in \R^n$ and let $f: M_1 \to M_2$ be a homomorphism between $v$-semistable modules $M_1$ and $M_2$. If $f$ is zero or an isomorphism the result follows immediately. Otherwise, we want to show that $\ker f$ and $\coker f$ are $v$-semistable. Consider the short exact sequences
\[ 0 \to \ker f \to M_1 \to \image f \to 0, \]
\[ 0 \to \image f \to M_2 \to \coker f \to 0. \]

By definition, the $v$-semistable modules $M_i$ satisfy $\langle v, \bdim M_i \rangle = 0$ and $\langle v, \bdim L \rangle \leq 0$ for any subobject $L$ of $M_i$ or equivalently $\langle v, \bdim N \rangle \geq 0$ for any quotient $N$ of $M_i$. So in particular we get 
\[ \langle v, \bdim \image f \rangle \geq \langle v, \bdim M_1 \rangle = 0, \]
\[ \langle v, \bdim \image f \rangle \leq \langle v, \bdim M_2 \rangle = 0. \]
Thus $\langle v, \bdim \image f \rangle = 0$. Furthermore, from the exact sequences it follows that
\[ \bdim \ker f = \bdim M_1 - \bdim \image f \quad \text{and} \quad \bdim \coker f = \bdim M_2 - \bdim \image f,\]
which tells us that
\[ \langle v, \bdim \ker f \rangle = \langle v, \bdim M_1 \rangle - \langle v, \bdim \image f \rangle = 0,\]
and similarly $\langle v, \bdim \coker f \rangle = 0$. Since subobjects of $\ker f$ are subobjects of $M_1$ and quotients of $\coker f$ are quotients of $M_2$, they satisfy the necessary inequalities. Thus $\ker f, \coker f \in \mods_{v}^{ss} A$. \\

For extensions, we assume that $X,Z \in \mods_{v}^{ss}A$ in the diagram below and want to show that $Y \in \mods_{v}^{ss} A$. Similar to above we have that $\langle v, \bdim Y \rangle = \langle v, \bdim X \rangle + \langle v, \bdim Z \rangle = 0$ because of exactness of the sequence. What is left to prove is that for all subobjects $L$ of $Y$ we have $\langle v, \bdim L \rangle \leq 0$, which can be seen from the following diagram:

\begin{equation}
\begin{tikzcd}
0 \arrow[r] & X \arrow[r,"f",hook] & Y \arrow[r,"g",two heads] & Z \arrow[r] & 0 \\
0 \arrow[r] & \ker h \arrow[u, hook] \arrow[r, hook] &\arrow[u, hook,"i"]  L \arrow[r,"h", two heads] & \image (gi) \arrow[u, hook]  \arrow[r] & 0 \\
& 0 \arrow[u] & 0 \arrow[u] & 0 \arrow[u]
\end{tikzcd}
 \label{eq:closedsubobj}
\end{equation}
For any injection $L \hookrightarrow Y$ we have $\langle v, \bdim L \rangle = \langle v, \bdim \ker h \rangle + \langle v, \bdim \image(gi) \rangle \leq 0$.
\end{proof} 

The following definition of a brick is essential in studying stability conditions and related notions.

\begin{defn} 
An object $B \in \mods A$ is called a \textit{brick} if its endomorphism algebra $\End_A(B)$ is a division ring i.e. every nonzero element has a multiplicative inverse.
\end{defn}

The first connection between stable modules and bricks is the following.

\begin{prop}\cite[Theorem 1]{Rudakov1997}
If $B$ is $v$-stable then $B$ is a brick.
\end{prop}

\begin{proof} Let $B$ be $v$-stable. Similar to the proof of the previous theorem, if we assume that $f: B \to B$ is not zero and not an isomorphism, then the following short exact sequences
\[ 0 \to \ker f \hookrightarrow B \twoheadrightarrow \image f \to 0, \]
\[ 0 \to \image f \hookrightarrow B \twoheadrightarrow \coker f \to 0 \]
imply
\[ \langle v, \bdim \image f \rangle > \langle v, \bdim B \rangle > \langle v, \bdim \image f \rangle,\]
which is a contradiction. Hence, $f$ is either 0 or an isomorphism, thus any morphism $f \in \End(B)$ is invertible and thus $\End_A(B)$ a division ring when $B$ is $v$-stable.
\end{proof}

We observe the following connection between $v$-stables and $v$-semistables.

\begin{lem}
The (relatively) simple modules in $\mods_v^{ss} A$ are exactly the $v$-stable objects.
\end{lem}
\begin{proof}
Let $M$ be a $v$-stable module, then $\langle v, \bdim L \rangle < 0$ for all nonzero proper submodules $L$ of $M$. Hence it cannot have nonzero $v$-semistable proper submodule $L$ which would have to satisfy $\langle v, \bdim L \rangle = 0$. Now assume that a $v$-semistable module $M$ does not have any nonzero proper submodules, then it trivially satisfies $\langle v, \bdim L \rangle < 0$ for all nonzero proper submodules $L$ and is therefore $v$-stable. 
\end{proof}

The following theorem by Rudakov \cite{Rudakov1997} is reminiscent of the ``Jordan-Hölder-Theorem'' for semistable modules because of the previous lemma.

\begin{thm} \cite[Theorem 3]{Rudakov1997}
Let $M \in \mods_v^{ss} A$. Then there exists a filtration
\[ 0 = M_0 \subset M_1 \subset \dots \subset M_t = M \]
such that $M_i / M_{i-1}$ is $v$-stable. Moreover, any two such filtrations have the same length and the multiset $\{ M_i /M_{i-1} \}$ is independent of the filtration chosen.
\end{thm}

\begin{proof}
    We only show the existence of such a filtration. Since the $v$-stable modules are relatively simple in $\mods_v^{ss}A$ the ``uniqueness'' of the filtration can be proven following arguments similar to those used to show the classical Jordan-Hölder Theorem. \\
    
    If $M$ is $v$-stable, then we may take the filtration $0 \subseteq M$ and we are done. 
    Assume that $M$ is $v$-semistable, but not $v$-stable, then there exists some nonzero proper submodule $L$ of $M$ such that $\langle v, \bdim L \rangle = 0$. Take $L$ to be a maximal submodule with that property, then $M/L$ satisfies
    \[ \langle v, \bdim M/L \rangle = \langle v, \bdim M \rangle - \langle v, \bdim L \rangle = 0.\]
    Moreover any submodule $N/L$ of $M/L$ corresponds to a submodule $N$ of $M$ containing $L$. Thus
    \[ \langle v, \bdim N/L \rangle = \underbrace{\langle v, \bdim N \rangle}_{< 0} - \underbrace{\langle v, \bdim L \rangle}_{=0} < 0.\]
    Then $M/L$ is $v$-stable. 
    If $L$ is $v$-stable, then we are done. Otherwise we repeat this process with $L$,
    which eventually ends since $\mods A$ is a length category.
\end{proof}

Baumann, Kamnitzer and Tingley \cite{BKT2011} showed that we can associate to every stability condition two torsion pairs in the following way. 

\begin{prop} \label{prop:TFalternatedef} \cite[Proposition 3.1]{BKT2011} Let $v \in \R^n$ be a stability conditon, then there exist torsion pairs $(\Tcal_v, \overline{\Fcal}_v)$ and $(\overline{\Tcal}_v, \Fcal_v)$ where
\begin{enumerate}
    \item $\Tcal_v = \{0 \} \cup \{ Y \in \mods A: \forall Y \to Z \to 0, \langle v, \bdim Z \rangle > 0  \}$,
    \item $\overline{\Tcal}_v = \{0 \} \cup \{ Y \in \mods A:\forall Y \to Z \to 0, \langle v, \bdim Z \rangle \geq 0 \}$,
    \item $\Fcal_v = \{ 0 \} \cup \{ Y \in \mods A: \forall 0 \to X \to Y ,  \langle v, \bdim X \rangle < 0\}$,
    \item $\overline{\Fcal}_v = \{ 0 \} \cup \{ Y \in \mods A:\forall 0 \to X \to Y, \langle v, \bdim X \rangle \leq 0 \}$.
\end{enumerate}
Moreover, $\mods_v^{ss}A = \overline{\Tcal}_v \cap \overline{\Fcal}_v$. 
\end{prop}

\begin{proof} 
By definition $\Tcal_v$, $\overline{\Tcal}_v$, $\Fcal_v$ and $\overline{\Fcal}_v$ are full subcategories of $\mods A$. First, let us use \cref{prop:TFalternatedef} to show that each of these is a torsion or torsion-free class. \\

Take $T \in \Tcal_v$ and consider a quotient $T'$ of $T$. Then any quotient $T''$ of $T'$ is also a quotient of $T$ and hence satisfies $\langle v, \bdim T'' \rangle > 0$ as required. Thus $T' \in \Tcal_v$. Furthermore if $T$ and $T'$ are in $\Tcal_v$ consider the following extension
\begin{equation*}
\begin{tikzcd}
0 \arrow[r] & T \arrow[r,"f"] \arrow[d] & M \arrow[r,"g"] \arrow[d,"p"] & T' \arrow[r] \arrow[d] & 0 \\
0 \arrow[r] & \image (fp)\arrow[r,"h"] \arrow[d]& N \arrow[d] \arrow[r] & \coker h  \arrow[r] \arrow[d] & 0 \\
& 0  & 0 & 0
\end{tikzcd}
\end{equation*}
where $M$ is an extension of $T'$ by $T$ and every quotient $N$ of $M$ satisfies $\langle v, \bdim N \rangle = \langle v, \bdim \image (fp) \rangle + \langle v, \bdim \coker h \rangle > 0$ since the right-hand side terms are both quotients of elements in $\Tcal_v$. Therefore $M \in \Tcal_v$ and thus $\Tcal_v$ is a torsion-class. \\

It follows in the same manner that $\overline{\Tcal}_v$ is a torsion-class. To show that $\Fcal_v$ and $\overline{\Fcal}_v$ are torsion-free classes we notice that subobjects of subobjects are subobjects and use a diagram similar to \ref{eq:closedsubobj} to show that they are closed under extensions. \\

All that is left to show is that these pairs are $\Hom$-orthogonal. Take $v \in \R^n$ and consider $(\Tcal_v, \overline{\Fcal}_v)$. Assume there is a morphism $f:T \to F$, then it factors through its image
\[ T \twoheadrightarrow \image f \hookrightarrow F .\]
Then immediately by the definition of $\Tcal_v$ and $\overline{\Fcal}_v$ we obtain $\langle v, \bdim \image f \rangle > 0$ and $\langle v, \bdim \image f \rangle \leq 0$ which is a contradiction i.e. $f = 0$ and $\Hom_A(T,F) = 0$. Similarly for $(\overline{\Tcal}_v, \Fcal_v)$.
\end{proof}

Instead of asking which modules $M$ are $v$-semistable for a given $v \in \R^n$, let us now consider for which $v \in \R^n$ a module $M$ is $v$-semistable. 
It turns out that the collection of vectors (i.e. stability conditions) for which a module is semistable, called the \textit{stabiliy space}, forms a rich structure.
We now introduce the main object of study of these notes: the \textit{wall-and-chamber structure} of an algebra.

\begin{defn}
Fix a nonzero module $M \in \mods A$, then the \textit{stability space} $\Dcal(M)$ of $M$ is 
\[ \Dcal(M) = \{ v \in \R^n: M \text{ is $v$-semistable}\} \subseteq \R^n. \]
We say $\Dcal(M)$ is a \textit{wall} if $\codim \Dcal(M) = 1$. A \textit{chamber} is an open connected component of 
\[\R^n \setminus \overline{\bigcup_{\substack{M \in \mods A \\ 0 \neq M}} \Dcal(M)}.\]
\end{defn}

The combination of all stability spaces $\Dcal(M)$ for indecomposable modules $M$ and the corresponding chambers they define is called the \textit{wall-and-chamber structure}. Because of the following result it is sufficient to calculate the stability spaces of the indecomposable modules to obtain the whole wall-and-chamber structure.

\begin{prop}
    Let $M, N \in \mods A$, then $\Dcal(M \oplus N) \subseteq \Dcal(M) \cap \Dcal(N)$.
\end{prop}
\begin{proof}
    Let $v \in \Dcal (M \oplus N)$, then by definition
    \[ 0 = \langle v, \bdim (M \oplus N) \rangle = \langle v, \bdim M + \bdim N \rangle= \langle v, \bdim M \rangle + \langle v, \bdim N \rangle.\]
    Since $M$ and $N$ are both subobjects and quotients of $M \oplus N$, the two terms on the right hand side of the equation above have to be equal to $0$. 
    Moreover, any submodule $L$ of $M$ or $N$ is also submodule of $M \oplus N$, therefore satisfies $\langle v, \bdim L \rangle \leq 0$. Hence $v \in \Dcal(M) \cap \Dcal(N)$.
\end{proof}

We illustrate this concept on the simple example of $A_2$. Since the quiver has 2 vertices, its wall-and-chamber structure lies in $\R^2$ and is easy to visualise.

\begin{exmp} \label{exmp:A2wac}
Let $A$ be the path algebra of the quiver $Q = \begin{tikzcd} 1 \arrow[r] & 2 \end{tikzcd}$. 
The Auslander-Reiten quiver of $A$ is given by
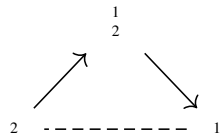
\[ \begin{tikzcd}[ampersand replacement=\&,column sep = 1.5em, row sep=2em]
\& {\tiny \begin{array}{cc} 1 \\2 \end{array}} \arrow[rd] \\
\tiny \begin{array}{c}  2 \end{array} \arrow[ru] \& \& \tiny \begin{array} {c} 1 \end{array} \arrow[ll,dashed, dash]
\end{tikzcd}\]
We compute the stability spaces of the indecomposable modules
\begin{align*}
    \Dcal(\tiny \begin{array}{c}  1 \end{array}) \normalsize &\coloneqq \{ v \in \R^2: \langle v, (1,0) \rangle = 0 \} = \{ (0,y):y \in \R\}, \\
    \Dcal(\tiny \begin{array}{c}  2 \end{array}) &\coloneqq \{ v \in \R^2: \langle v, (0,1) \rangle = 0 \} = \{ (x,0): x \in \R \}, \\
    \Dcal(\tiny{\begin{array}{cc} 1 \\2 \end{array}}\normalsize) &\coloneqq \{ v \in \R^2: \langle v, (1,1) \rangle = 0 \text{ and }\langle v,(0,1)\rangle \leq 0 \} = \{ (x,-x): 0 \leq x \in \R\}.
\end{align*}
where the last line contains two conditions, since in contrast to the two above there also exists a non-trivial submodule ${\tiny \begin{array}{c}  2 \end{array}} \hookrightarrow {\tiny \begin{array}{cc} 1 \\2 \end{array}}$. Therefore the wall-and-chamber structure of $A_2$ is the given in \cref{fig:wacA2}.
\begin{figure}[ht!]
\centering
\begin{tikzpicture}
  \draw[-] (-3, 0) -- (3, 0) node[right] {$\Dcal({\tiny \begin{array}{c}  2 \end{array}})$};
  \draw[-] (0, -3) -- (0, 3) node[above] {$\Dcal({\tiny \begin{array}{c}  1 \end{array}})$};
  \draw[-] (0,0) -- (2.5,-2.5) node[right] {$\Dcal(\tiny{\begin{array}{cc} 1 \\2 \end{array}}\normalsize)$};
\end{tikzpicture}
\caption{Wall-and-chamber structure of $A_2$}\label{fig:wacA2}
\end{figure}
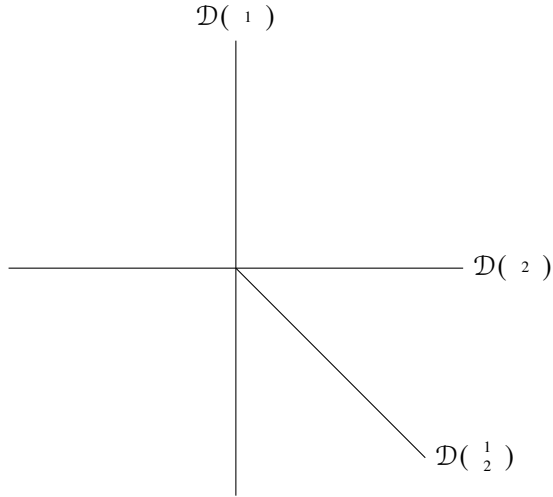
\end{exmp}

\begin{exmp}
Let $A= K
Q$ where $Q$ is the Kronecker quiver $Q= \begin{tikzcd} 1 \arrow[r,shift left] \arrow[r, shift right] & 2 \end{tikzcd}$. 
In this example we assume that $K$ is algebraically closed.
Its Auslander-Reiten quiver may be illustrated in the following way, where $\mathcal{R}$ corresponds to its regular components, see \cite[Section~VIII.2]{bluebookV1}.
\[ \begin{tikzcd}[ampersand replacement=\&,column sep = 1.5em, row sep=1em]
\& {\tiny \begin{array}{c} 1 \\22 \end{array}} \arrow[rdd, shift left]\arrow[rdd, shift right] \arrow[r,dash,dashed,end anchor={[xshift=1ex]east}]\& \phantom{1} \& \& \tau({\tiny \begin{array}{c}  1 \end{array}})\arrow[rdd, shift left]\arrow[rdd, shift right] \&  \& {\tiny \begin{array}{c}  1 \end{array}} \arrow[ll,dash,dashed]\\
\& \& \& \mathcal{R} \arrow[ru, dotted, dash]\\
{\tiny \begin{array}{c}  2 \end{array}} \arrow[ruu,shift left] \arrow[ruu,shift right] \&  \& \tau^{-1}({\tiny \begin{array}{c}  2 \end{array}}) \arrow[ru, dotted, dash]\arrow[ll,dashed, dash]  \&  \&\phantom{1} \& {\tiny \begin{array}{c} 11 \\2 \end{array}} \arrow[l, dash, dashed,end anchor={[xshift=-1ex]east}] \arrow[ruu, shift left] \arrow[ruu, shift right]
\end{tikzcd}\] 

These modules have the following dimension vectors:
\[ \bdim \tau^{-m}({\tiny \begin{array}{c}  2 \end{array}}) = \begin{pmatrix} 2m & 2m+1 \end{pmatrix} \quad \text{and} \quad \bdim \tau^{-m}({\tiny \begin{array}{c} 1 \\22 \end{array}}) = \begin{pmatrix} 2m+1 & 2m+2 \end{pmatrix},\]
\[ \bdim \tau^{m}({\tiny \begin{array}{c}  1 \end{array}}) = \begin{pmatrix} 2m+1 & 2m \end{pmatrix} \quad \text{and} \quad \bdim \tau^{-m}({\tiny \begin{array}{c} 11 \\2 \end{array}}) = \begin{pmatrix} 2m+2 & 2m+1 \end{pmatrix}. \]
Moreover, every indecomposable $\mathcal{R}$  is of the form $\tiny{\begin{array}{cc} 1 \\2 \end{array}}(d,\lambda) = \normalsize \{\begin{tikzcd} K^d \arrow[r,shift left,"J_d(\lambda)"] \arrow[r,shift right,"Id", swap] & K^d \end{tikzcd}\}$ where $d \in \mathbb{N}$, $\lambda \in \Pp^1(K)$ and $J_d(\lambda)$ is the Jordan block of size $d$ having $\lambda$ in the diagonal.
Therefore we have
\[ \Dcal({\tiny \begin{array}{c}  2 \end{array}}) = (x,0):x \in \R \} \quad \text{and} \quad \Dcal({\tiny\begin{array}{cc} 1 \end{array}}) = \{ (0,y):y \in \R\}\]
and for $m \geq 1$ we obtain
\begin{align*} \Dcal(\tau^{-m}({\tiny \begin{array}{c}  2 \end{array}})) &= \left\{ v \in \R^2: \bigg\langle v, \begin{bmatrix} 2m \\ 2m+1 \end{bmatrix} \right\rangle = 0, &\text{and } \left\langle v, \begin{bmatrix} 2i \\ 2i +1 \end{bmatrix}\right \rangle \leq 0 \quad \forall 0 \leq i < m,\text{ } \\
&&\text{and } \left\langle v, \begin{bmatrix} 2i+1 \\ 2i +2 \end{bmatrix}\right \rangle \leq 0 \quad \forall 0 \leq i < m \bigg\} \\
&= \{ ((2m+1)x,-2mx): 0 \leq x \in \R \} &
\end{align*}
Similarly,
\begin{align*}
    \Dcal(\tau^{-m}({\tiny \begin{array}{c} 1 \\22 \end{array}})) &= \{ (2(m+1)x,-(2m+1)x): 0 \leq x \in \R\}, \\
    \Dcal(\tau^m({\tiny \begin{array}{c}  1 \end{array}})) &= \{(2mx,-(2m+1)x):0 \leq x \in \R\}, \\
    \Dcal(\tau^{m}({\tiny \begin{array}{c} 11 \\2 \end{array}})) &= \{ ((2m+1)x,-2(m+1)x): 0 \leq x \in \R\}.\\
    \Dcal(\tiny{\begin{array}{cc} 1 \\2 \end{array}}(d,\lambda)) &= \{ ((x,-x): 0 \leq x \in \R\}.
\end{align*}
The wall-and-chamber structure is pictured in \cref{fig:wacKronecker}.
\begin{figure}[ht!]
\centering
\begin{tikzpicture}
  \draw[-] (-3, 0) -- (3, 0) node[right] {$\Dcal({\tiny \begin{array}{c} 2 \end{array}})$};
  \draw[-] (0, -3) -- (0, 3) node[above] {$\Dcal({\tiny \begin{array}{c} 1 \end{array}})$};
  \draw[-] (0,0) -- (3,-1.5) node[right] {$\Dcal(\tiny{\begin{array}{cc} 1 \\22 \end{array}}\normalsize)$};
  \draw[-] (0,0) -- (3,-2) node[right] {$\Dcal(\tau^{-1}({\tiny \begin{array}{c}  2 \end{array}}))$};
  \draw[-] (0,0) -- (3,-9/4);
  \draw[-] (0,0) -- (3,-12/5);
  \draw[-] (0,0) -- (3,-2.5);
  \draw[dotted] (0,0) -- (3,-3) node[below right] {$\Dcal(\tiny{\begin{array}{cc} 1 \\2 \end{array}}\normalsize(d,\lambda))$};
  \draw[-] (0,0) -- (2.5,-3);
  \draw[-] (0,0) -- (12/5,-3);
  \draw[-] (0,0) -- (9/4,-3);
  \draw[-] (0,0) -- (2,-3);
  \draw[-] (0,0) -- (1.5,-3) node[below] {$\Dcal(\tiny{\begin{array}{cc} 11 \\2 \end{array}}\normalsize)$};
\end{tikzpicture}
\caption{Wall-and-chamber structure of the Kronecker quiver}\label{fig:wacKronecker}
\end{figure}
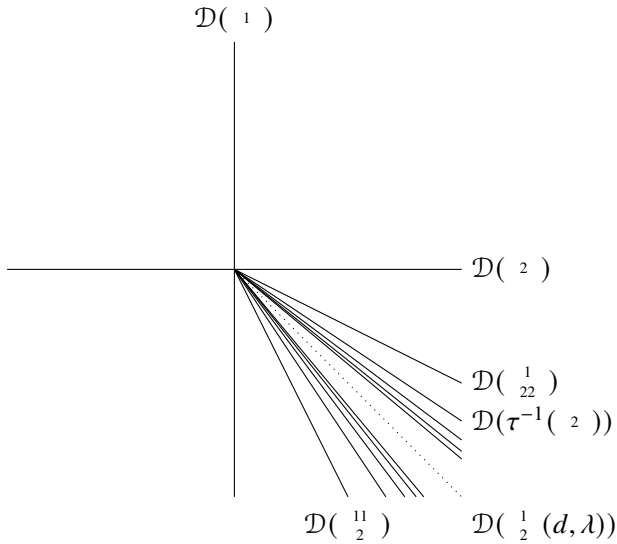
\end{exmp}

To conclude this section we calculate the wall-and-chamber structure of an algebra of rank 3 and demonstrate how we are still able to obtain an understandable image by using an stereographic projection.

\begin{exmp} \label{exmp:A3loop}
    Let $A$ be the path algebra over the quiver $Q = \begin{tikzcd} 1 \arrow[r] & 2 \arrow[r] & 3 \arrow[ll,bend right] \end{tikzcd}$ modulo the square of the arrow ideal. The Auslander-Reiten quiver is given by
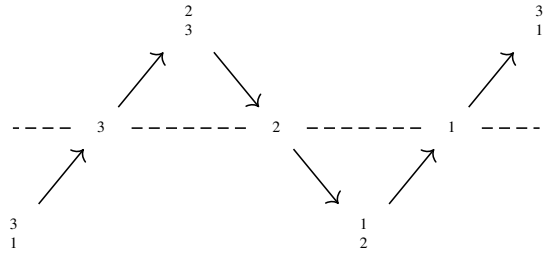
\begin{equation*} \begin{tikzcd}[ampersand replacement=\&,column sep = 1em, row sep=2em]
\& \& {\tiny\begin{array}{cc} 2 \\ 3 \end{array}} \arrow[rd]  \& \& \&\&  {\tiny\begin{array}{cc} 3 \\ 1\end{array}} \\
\phantom{ }\& {\tiny \begin{array}{c} 3 \end{array}} \arrow[ru] \arrow[l, dashed,dash,end anchor={[xshift=1ex]west}]\& \& {\tiny \begin{array}{c} 2 \end{array}} \arrow[rd] \arrow[ll,dashed,dash]\& \& {\tiny \begin{array}{c} 1 \end{array}} \arrow[ru]\arrow[ll,dashed,dash] \arrow[r,dash,dashed, end anchor={[xshift=-1ex]east}]\& \phantom{ }  \\
{\tiny\begin{array}{cc} 3 \\1 \end{array}} \arrow[ru] \& \& \& \& {\tiny\begin{array}{cc} 1 \\ 2 \end{array}} \arrow[ru] 
\end{tikzcd}
\end{equation*}

Therefore we may calculate the stability spaces of indecomposables to obtain

\[ \Dcal({\tiny \begin{array}{c} 1 \end{array}})= \left\{ \begin{pmatrix} 0 \\ y \\ z \end{pmatrix}: y, z \in \R \right\}, \qquad  \Dcal({\tiny \begin{array}{c} 2 \end{array}})= \left\{ \begin{pmatrix} x \\ 0 \\ z \end{pmatrix} : x, z \in \R \right\},\]
\[\Dcal({\tiny \begin{array}{c} 3 \end{array}})= \left\{ \begin{pmatrix} x \\ y \\ 0 \end{pmatrix} : x, y \in \R \right\}, \qquad  \Dcal({\tiny \begin{array}{cc} 1 \\2 \end{array}})= \left\{ \begin{pmatrix} x \\ -x \\ z \end{pmatrix} : x \geq 0, z \in \R \right\},\]
\[\Dcal({\tiny \begin{array}{cc} 2 \\3 \end{array}})= \left\{ \begin{pmatrix} x\\y \\ -y \end{pmatrix} : y \geq 0, x \in \R \right\}, \qquad \Dcal({\tiny \begin{array}{cc} 3 \\1 \end{array}})= \left\{ \begin{pmatrix} -x\\y \\ x \end{pmatrix} : x \geq 0, y \in \R \right\}.\]

We may visualise these in $\R^3$ in the following way:

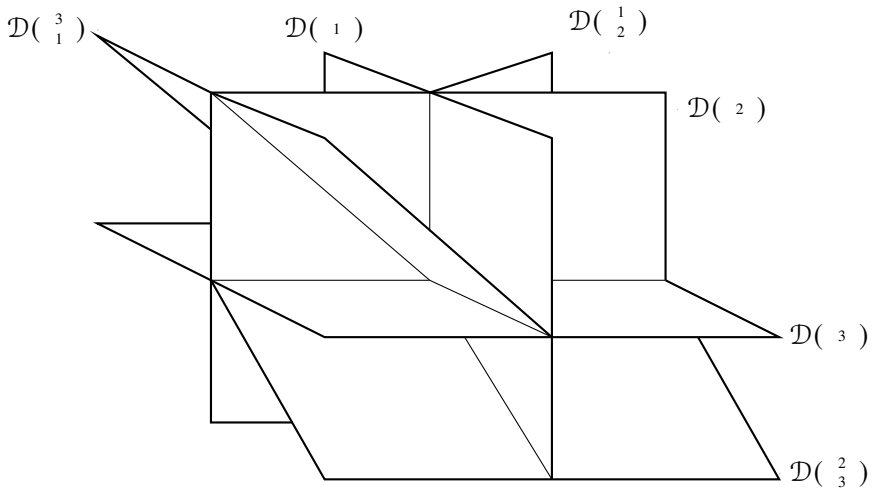
\[
    \begin{tikzpicture}[thick,scale=1.5]
    \draw[thin]  (1,-2.25) -- (-0.075,-0.5);
    \draw[]  (-2,-0.5) -- (-1,-2.25) -- (3,-2.25) -- (2.275,-0.98125);
    \draw[fill=white]   (2,-0.5) -- (3,-1) -- (-1,-1) -- (-3,0) -- (-2,0);
    \draw[]  (-1,1.15) -- (-1,1.5) -- (1,0.75) -- (1,-2.25);
    \draw[]  (-1.28,-1.75) -- (-2,-1.75) -- (-2,1.15) -- (2,1.15) -- (2,-0.5);
    \draw[]  (-0.075,1.15) -- (1, 1.5) -- (1, 1.15);
    \draw[]  (1,-1) -- (-1,0.75) -- (-2,1.15) -- (-3,1.65) -- (-2,0.825);
    \draw[thin]  (-0.075,-0.07) -- (-0.075,1.15);
    \draw[thin]  (-0.075,-0.5) -- (-2,-0.5);
    \draw[thin]  (-0.075,-0.5) -- (1,-1);
    \draw[thin]  (1,-0.5) -- (2,-0.5);
    \draw[thin]  (-0.075,-0.5) -- (-2,1.15);
    \filldraw[black] (-3,1.7) circle (0pt) node[left]{$\Dcal({\tiny \begin{array}{c} 3\\1 \end{array}})$};
    \filldraw[black] (3,-2.2) circle (0pt) node[right]{$\Dcal({\tiny \begin{array}{c} 2 \\ 3 \end{array}})$};
    \filldraw[black] (1.5,1.5) circle (0pt) node[above]{$\Dcal({\tiny \begin{array}{c} 1 \\ 2 \end{array}})$};
    \filldraw[black] (2.1,1) circle (0pt) node[right]{$\Dcal({\tiny \begin{array}{c} 2 \end{array}})$};
    \filldraw[black] (3,-1) circle (0pt) node[right]{$\Dcal({\tiny \begin{array}{c} 3 \end{array}})$};
    \filldraw[black] (-1,1.5) circle (0pt) node[above]{$\Dcal({\tiny \begin{array}{c} 1 \end{array}})$};
    \end{tikzpicture}
    \]

To get a better understanding we perform a stereographic projection, i.e. we first take the intersection of $\bigcup_{\substack{M \in \mods A\\0 \neq M}}\Dcal(M)$ with the unit sphere centered at the origin to obtain \cref{fig:wacintersect}.
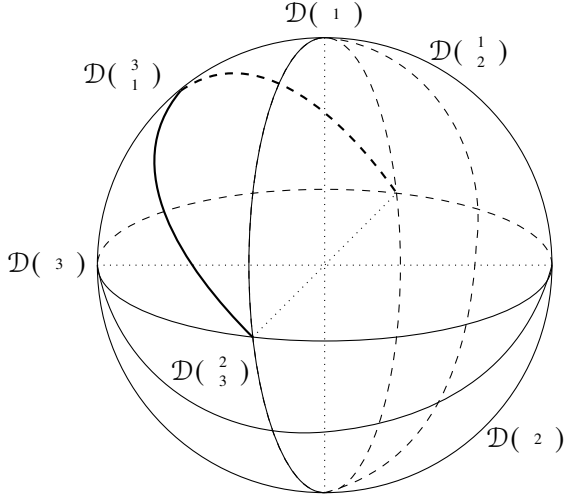
\begin{figure}[ht!]
\centering
    \begin{tikzpicture}
    \draw[color=black] (-3,0) arc
    [
        start angle=180,
        end angle=360,
        x radius=3cm,
        y radius =1cm
    ];
    \draw[dashed, color=black] (3,0) arc
    [
        start angle=0,
        end angle=180,
        x radius=3cm,
        y radius =1cm
    ];
    \draw[color=black] (0,3) arc
    [
        start angle=90,
        end angle=270,
        x radius=1cm,
        y radius =3cm
    ];
    \draw[dashed, color=black] (0,-3) arc
    [
        start angle=270,
        end angle=-90,
        x radius=1cm,
        y radius =3cm
    ];
    \draw[color=black] (0,0) circle [radius=3];
    \draw [black] (3,0) to[out=-105,in=5] (0,-2.2) to[out=-175,in=-80] (-3,0);
    \draw [black,dashed] (0,3) to[out=-5,in=82] (2,0.25) to[out=-93,in=15] (0,-3);
    \draw [black,thick ] (-0.95,-0.95) to[out=135,in=-130] (-1.9,2.3);
    \draw [black,thick, dashed] (-1.9,2.3) to[out=35,in=125] (0.95,0.95);
    \draw[black,dotted] (3,0)--(-3,0);
    \draw[black,dotted] (0,3)--(0,-3);
    \draw[black,dotted] (-0.95,-0.95)--(0.95,0.95);
    \filldraw[black] (-3,0) node[left]{$\Dcal({\tiny \begin{array}{c} 3 \end{array}})$};
    \filldraw[black] (2,-2.3) node[right]{$\Dcal({\tiny \begin{array}{c} 2 \end{array}})$};
    \filldraw[black] (0,3) node[above]{$\Dcal({\tiny \begin{array}{c} 1 \end{array}})$};
    \filldraw[black] (1.25,2.8) node[right] {$\Dcal({\tiny \begin{array}{cc} 1 \\2 \end{array}})$};
    \filldraw[black] (-1.5,-1.8) node[above] {$\Dcal({\tiny \begin{array}{cc} 2 \\3 \end{array}})$};
    \filldraw[black] (-2,2.5) node[left] {$\Dcal({\tiny \begin{array}{cc} 3 \\1 \end{array}})$};
    \end{tikzpicture}
    \caption{Intersection of the wall-and-chamber stucture with the unit sphere}\label{fig:wacintersect}
    \end{figure}
    If we project from the point $(1,1,1)$ we obtain \cref{fig:steproj1}.
    \begin{figure}[ht!]
    \centering
    \begin{tikzpicture}[thick,scale=0.8, every node/.style={transform shape}]
        \draw [black] (-4,-3.4641) arc [
            start angle=-145,
            end angle=-335,
            x radius=3.47cm,
            y radius=3.47cm
        ];
        \draw [black] (4,-3.4641) arc [
            start angle=-25,
            end angle=-215,
            x radius=3.47cm,
            y radius=3.47cm
        ];
        \draw [black] (0,3.4641) arc [
            start angle=-265,
            end angle=-455,
            x radius=3.47cm,
            y radius =3.47cm
        ];
        \draw[color=black,thick] (0,-3.46410161514) circle [radius=4];
        \draw[color=black,thick] (2,0) circle [radius=4];
        \draw[color=black,thick] (-2,0) circle [radius=4];
        \filldraw[black] (-6,0) node[left]{$\Dcal({\tiny \begin{array}{c} 3 \end{array}})$};
        \filldraw[black] (0,-7.5) node[below]{$\Dcal({\tiny \begin{array}{c} 2 \end{array}})$};
        \filldraw[black] (6.75,0) node[black]{$\Dcal({\tiny \begin{array}{c} 1 \end{array}})$};
        \filldraw[black] (0,-6) node[right] {$\Dcal({\tiny \begin{array}{cc} 1 \\2 \end{array}})$};
        \filldraw[black] (-2.5,2) node[left] {$\Dcal({\tiny \begin{array}{cc} 2 \\3 \end{array}})$};
        \filldraw[black] (3.5,1.5) node[right] {$\Dcal({\tiny \begin{array}{cc} 3 \\1 \end{array}})$};
    \end{tikzpicture}
    \caption{Stereographic projection of the wall-and-chamber structure}\label{fig:steproj1}
    \end{figure}
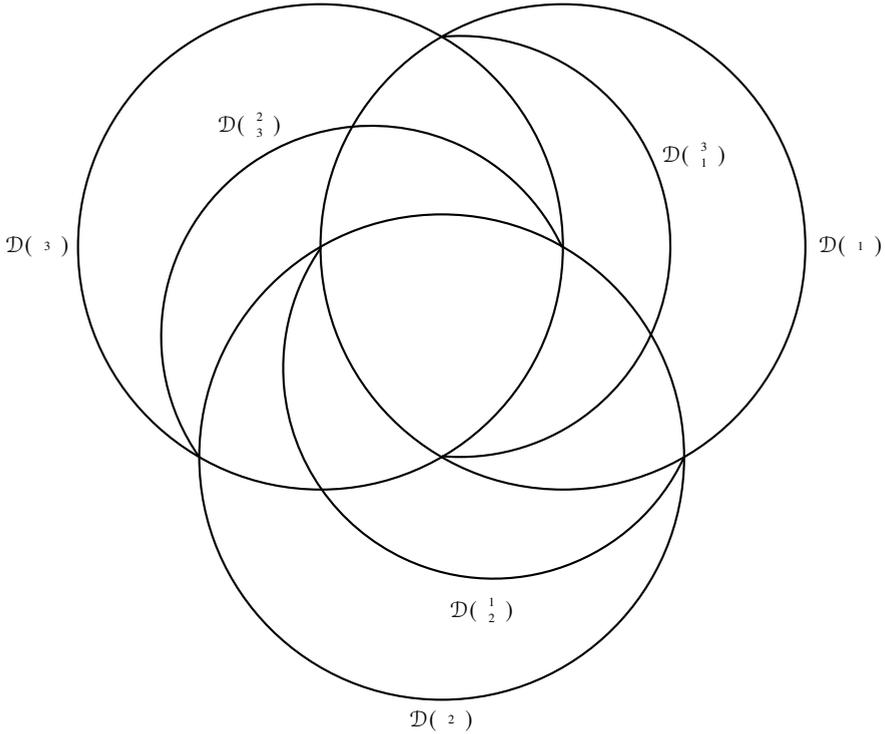
    Note that the exterior of \cref{fig:steproj1} is also a chamber and we find that the wall-and-chamber structure for $Q$ consists of 14 chambers. 
    \end{exmp}

\section{$\tau$-tilting theory}
The area of $\tau$-tilting theory is a recent development in the representation theory of finite dimensional algebras. 
The name is a combination of (classical) tilting theory and Auslander-Reiten theory, where $\tau$ represents the Auslander-Reiten translation. 
It was first introduced in the early 2010s by Adachi, Iyama and Reiten \cite{AIR2014} and has since become an active area of research. 
Many connections with other mathematical subfields have been established and continue to be discovered. 
This new theory can be viewed as a completion of classical tilting theory with respect to mutations. 
We begin by introducing some definitions.

\begin{defn} \cite{AIR2014}
Let $T,P \in \mods A$, where $P$ is projective. Then 
\begin{enumerate}
    \item $T$ is $\tau$\textit{-rigid} if $\Hom_A(T,\tau T)=0$; 
    \item $T$ is $\tau$\textit{-tilting} if it is $\tau$-rigid and $|T| = |A|$; 
    \item a pair $(T,P)$ is $\tau$\textit{-rigid} if $T$ is $\tau$-rigid  and $\Hom_A(P,T) =0$; 
    \item a $\tau$-rigid pair $(T,P)$ is $\tau$\textit{-tilting} if $|T|+|P|=|A|=n$.
\end{enumerate}
\end{defn}

Sometimes, the term \textit{support $\tau$-tilting module} is used to describe a $\tau$-rigid module $T$ which is part of a $\tau$-tilting pair $(T,P)$. For comparison, let us give the definition of a (classical) tilting module.

\begin{defn}
    Let $T \in \mods A$. Then $T$ is called \textit{tilting} if 
    \begin{enumerate}
        \item $\pd T \leq 1$, i.e the projective dimension of $T$ is less than or equal to 1, and
        \item $\Ext_A^1(T,T) = 0$, i.e. $T$ is \textit{rigid}, and
        \item[(3a)] there exists a short exact sequence $0 \to A \to T' \to T'' \to 0$, where $T', T'' \in \add T$.
    \end{enumerate}
\end{defn}

At first glance there may not seem to be many similarities between $\tau$-tilting modules and tilting modules. However, by \cite[Corollary VI.4.4]{bluebookV1}, the third condition is equivalent to the following:
\begin{enumerate}
    \item[(3b)] $|T| = |A|$.
\end{enumerate}

Moreover, \cite[Proposition 5.8]{AuslanderSmalo1980} states that $\Hom_A(T, \tau T)= 0$ implies $\Ext_A^1(T,T)=0$ and \cite[Corollary IV.2.14]{bluebookV1} implies that the converse holds when $\pd T \leq 1$. This means that, when the projective dimension of a module $M$ is at most one, a module $M$ is tilting if and only if it is $\tau$-tilting. Therefore, one can view $\tau$-tilting as a possible generalisation of tilting theory. 
The following result is sometimes called Skowroński's Lemma.

\begin{lem} \cite[Proposition 1.3]{AIR2014} \label{lem:rigidindecsummands}
If $(T,P)$ is $\tau$-rigid then $|T|+|P| \leq n$.
\end{lem}
\begin{proof}Let $e$ be an idempotent of $A$ such that $\add P = \add Ae$ and write $B= A/\langle e \rangle$. Then \cite[Lemma 2.1b]{AIR2014} implies that $\Hom_A(T, \tau T)=0$ if and only if $\Hom_{B} (T, \tau_{B} T)=0$. It follows from \cite[Theorem 5.10]{AuslanderSmalo1980} and from \cite{Smalo1984} that $|T| \leq |B/\textnormal{ann}(T)| \leq |B| = |A|-|P| = n -|P|$. 
\end{proof}

A natural question to ask is what the minimal torsion class containing a given module $M \in \mods A$ is. To answer this, we must first introduce some new notions. Recall, that for a module $M$ we denote by $\Fac M$ the full subcategory of $\mods A$ of all quotient modules of finite direct sums of copies of $M$. Now, if $\Xcal$ is a subcategory of $\mods A$, then the full subcategory $\Filt(\Xcal)$ is defined by
\[ \Filt(\Xcal) \coloneqq \{ X \in \mods A: \exists 0 = X_0 \subset X_1 \subset \dots \subset X_t = X \text{ such that } X_{i+1}/X_{i} \in \Xcal \}.\]

We are now able to answer the question above with the following well-known result first stated in \cite{DIJ2019} and proved in \cite{Thomas2021}.

\begin{prop}
    Let $M \in \mods A$, the minimal torsion class containing $M$ is $\Filt(\Fac M)$.
\end{prop}
\begin{proof} Let us first show that $\Filt(\Fac M)$ is indeed a torsion class by showing it is closed under quotients and extensions. \\

Take $N \in \Filt (\Fac M)$ with composition series $0 = N_0 \subset N_1 \subset \dots \subset N_t = N$ such that $N_{i+1}/N_i \in \Fac M$. Consider a quotient of $N$ denoted by $N' \coloneqq N /L$. Then we obtain the following filtration of $N'$,
\[ 0 = N_0' \subset N_1' \subset \dots \subset N_r' = N',\]
where $N_i' \coloneqq (N_i + L)/L$. Thus 
\begin{align*}
N_{i+1}'/N_i' & = \left( ( N_{i+1} + L)/L \right) / \left( (N_i + L) / L \right) \\
& \cong (N_{i+1} +L) / (N_i + L),
\end{align*}
where we made use of the third isomorphism theorem to obtain the second line. Hence there is a well-defined homomorphism $N_{i+1} / N_i \to (N_{i+1} +L) / (N_i + L)$ given by $n + N_i \mapsto n + (N_i + L)$ which is surjective. Therefore $N_{i+1}'/N_i'$ is a quotient of $N_{i+1}/N_i$ and thus also an element of $\Fac M$. So $N' \in \Filt(\Fac M)$. \\

Now assume we have $N, N' \in \Filt(\Fac M)$, and consider the short exact sequence $0 \to N \to E \to N' \to 0$. Then $N' \cong E/N$ and submodules of $N'$ are in bijective correspondence with submodules of $E$ containing $N$. Let $N \subseteq L_i \subseteq E$ be the submodules corresponding to the filtration $0 = N_0' \subset N_1' \subset \dots \subset N_t' = N'$ of $N'$ such that $N_i' \cong L_i/N$. Then we can obtain a filtration for $E$ by taking
\[ 0 = N_0 \subset N_1 \subset \dots \subset N_s = N = L_0 \subset L_1 \subset \dots \subset L_t = E.\]
Clearly the composition factors of the filtration of $N$ are in $\Fac M$ by definition. And the components of the right half of the filtration satisfy
\begin{align*}
L_{i+1} / L_i  & \cong (L_{i+1} / N) / (L_i /N) \\
& \cong N_{i+1}'/N_i',
\end{align*}
which is in $\Fac M$ as it is a composition factor of the filtration of $N'$. Thus $E \in \Filt(\Fac M)$. This is the smallest torsion class because every element is an iterated extension of elements in $\Fac M$, which must be contained in every torsion class containing $M$.
\end{proof}

However, $\Fac(\Filt(\add T))$ need \textit{not} be a torsion class since it might not be closed under extensions. 

\subsection{Finiteness conditions on subcategories}

In this subsection we show a close connection between $\tau$-tilting theory and a particular type of torsion classes in $\mods A$. We begin by introducing two definitions.

\begin{defn}
    Let $\Xcal \subseteq \mods A$ be a full subcategory. Given a module $M \in \mods A$, a \textit{right $\Xcal$-approximation} of $M$ is a map $f_M: X_M \to M$ with $X_M \in \Xcal$ such that for any map $g: Y \to M$ with $Y \in \Xcal$, there is a map $g': Y \to X$ such that the following diagram commutes
    \[\begin{tikzcd} X_M \arrow[r, "f_M"] & M \\ 
    Y \arrow[ru, "\forall g", swap] \arrow[u, "\exists g'", dashed]
    \end{tikzcd}\]
    In other words, every map $g: Y \to M$ factors through the map $f_M: X_M \to M$. Dually, a \textit{left $\Xcal$-approximation} of $M$ is a map $g_M: M \to X_M$ with $X_M \in \Xcal$ such that for any map $h: M \to Y$ with $Y \in \Xcal$, there is a map $h': X_M \to Y$ such that the following diagram commutes
    \[\begin{tikzcd} M \arrow[dr, "\forall h"]\arrow[r, "g_M"] & X_M \arrow[d, "\exists h'", dashed]\\ 
    & Y 
    \end{tikzcd}\]
\end{defn}

\begin{defn}
    We say that a full subcategory $\Xcal \subseteq \mods A$ is \textit{contravariantly finite (in $\mods A$)} if every $M \in \mods A$ admits a right $\Xcal$-approximation. Dually $\Xcal$ is \textit{covariantly finite} if every $M \in \mods A$ admits a left $\Xcal$-approximation. The subcategory $\Xcal$ is called \textit{functorially finite} if it is both contravariantly and covariantly finite.
\end{defn}

It turns out that we have been studying a particular example of such subcategories.

\begin{prop}
    Every torsion class in $\mods A$ is contravariantly finite. 
\end{prop}
\begin{proof}
    Let $\Tcal$ be a torsion class such that $\Fcal = \Tcal^\perp$ is the torsion-free class. Take any module $M \in \mods A$. Then by the definition of torsion pairs, there exists a canonical short exact sequence
    \[ 0 \to tM \xrightarrow{f_M} M \xrightarrow{g_M} fM \to 0\]
    where $tM \in \Tcal$. For any $T \in \Tcal$ consider the morphism $h: T \to M$, then since $fM \in \Fcal$, $g_M \circ h = 0$. So, $\image h \in \ker {g_M} = tM$. Hence $h$ factors through $tM$ like desired.
\end{proof}

Dually, one can show that any torsion-free class in $\mods A$ is covariantly finite. The following theorem by Auslander and Smalø \cite{AuslanderSmalo1980} lets us characterise the modules $T \in \mods A$ such that $\Fac T$ is a (functorially finite) torsion class. 

\begin{thm} \cite[Theorem 5.10]{AuslanderSmalo1980} 
    Let $T \in \mods A$. Then $\Fac T$ is a torsion class if and only if $T$ is $\tau$-rigid. In this case, $\Fac T$ is functorially finite. Moreover, every functorially finite torsion class arises this way.
\end{thm}

The previous classification of functorially finite torsion classes was refined by Adachi, Iyama and Reiten in \cite{AIR2014} using the notion of $\tau$-tilting pair. This is considered as one of the more fundamental results in $\tau$-tilting theory.

\begin{thm} \cite[Theorem 2.7]{AIR2014} 
    There is a one-to-one correspondence
    \begin{align*}
         \{ \text{$\tau$-tilting pairs} \} &\longleftrightarrow \{ \text{functorially finite torsion classes} \}. \\
        (T,P) &\longmapsto\quad \Fac T 
    \end{align*}
\end{thm}

A natural question in representation theory is to consider the algebras having a finite number of objects with a certain property. 
An algebra $A$ is said to be \textit{$\tau$-tilting finite} if there are finitely many basic $\tau$-tilting pairs in $\mods A$. 
The following characterisation of $\tau$-tilting finite algebras was given by Demonet, Iyama and Jasso in \cite{DIJ2019}.

\begin{thm} \cite[Theorem 3.8]{DIJ2019} 
Let $A$ be a finite dimensional algebra. The following are equivalent:
\begin{enumerate}
    \item $A$ is $\tau$-tilting finite.
    \item Every torsion class in $\mods A$ is functorially finite.
    \item Every torsion-free class in $\mods A$ is functorially finite.
\end{enumerate}
In this case, there are only finitely many torsion and torsion-free classes in $\mods A$.
\end{thm}

Recall, that for a torsion class $\Tcal$ we say $X \in \Tcal$ is \textit{$\Ext$-projective} if $\Ext_A^1(X, \Tcal)=0$. Denote by $P(\Tcal)$ the direct sum of one copy of each indecomposable $\Ext$-projective object in $\Tcal$ up to isomorphism. Then, the inverse bijection is given by sending a functorially finite torsion class $\Tcal$ to $P(\Tcal)$. Let us demonstrate this bijection using the simple case of $A_2$ which has exactly 5 $\tau$-tilting pairs.
\begin{exmp} \label{exmp:A2}
    Let $Q= \begin{tikzcd} 1 \arrow[r] & 2 \end{tikzcd}$. The Auslander-Reiten quiver given by
    
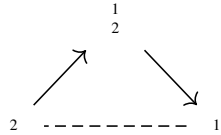
\[ \begin{tikzcd}[ampersand replacement=\&,column sep = 1.5em, row sep=2em]
    \& {\tiny \begin{array}{cc} 1 \\2 \end{array}} \arrow[rd] \\
    {\tiny \begin{array}{c} 2 \end{array}} \arrow[ru] \& \& {\tiny \begin{array}{c} 1 \end{array}} \arrow[ll,dashed, dash]
    \end{tikzcd}\]
Let us calculate $\Fac M$ for all $M \in \mods A$ and relate them to $\tau$-tilting pairs.
\[ \Fac({\tiny \begin{array}{c} 1 \end{array}}) = \add \{ {\tiny \begin{array}{c} 1 \end{array}} \} \quad \longleftrightarrow \quad ({\tiny \begin{array}{c} 1 \end{array}},{\tiny \begin{array}{c} 2 \end{array}}) \]
\[ \Fac(\tiny{\begin{array}{cc} 1 \\2 \end{array}}) \normalsize = \add \{ \tiny{\begin{array}{cc} 1 \\2 \end{array}} \normalsize\oplus {\tiny \begin{array}{c} 1 \end{array}} \} = \Fac( \tiny{\begin{array}{cc} 1 \\2 \end{array}} \normalsize\oplus {\tiny \begin{array}{c} 1 \end{array}} ) \quad \longleftrightarrow \quad ( \tiny{\begin{array}{cc} 1 \\2 \end{array}} \normalsize\oplus {\tiny \begin{array}{c} 1 \end{array}} , {\tiny\begin{array}{cc} 0 \end{array}}) \]
\[ \Fac({\tiny \begin{array}{c} 2 \end{array}}) = \add\{{\tiny \begin{array}{c} 2 \end{array}} \} \quad \longleftrightarrow \quad ({\tiny \begin{array}{c} 2 \end{array}}, \tiny{\begin{array}{cc} 1 \\2 \end{array}})\]
\[ \Fac({\tiny \begin{array}{c} 2 \end{array}} \oplus \tiny{\begin{array}{cc} 1 \\2 \end{array}}) = \add\{{\tiny \begin{array}{c} 2 \end{array}} \oplus \tiny{\begin{array}{cc} 1 \\2 \end{array}} \oplus 1\} = \mods A \quad \longleftrightarrow \quad ({\tiny \begin{array}{c} 2 \end{array}} \oplus \tiny{\begin{array}{cc} 1 \\2 \end{array}},{\tiny\begin{array}{cc} 0 \end{array}})\]
\[ \Fac({\tiny\begin{array}{cc} 0 \end{array}}) = \add \{ {\tiny\begin{array}{cc} 0 \end{array}}\} \quad \longleftrightarrow \quad ({\tiny\begin{array}{cc} 0 \end{array}},{\tiny \begin{array}{c} 2 \end{array}} \oplus \tiny{\begin{array}{cc} 1 \\2 \end{array}}). \]
So far we have only considered torsion classes $\Fac T$ arising from a $\tau$-rigid object $T$, however the module ${\tiny \begin{array}{c} 1 \end{array}} \oplus {\tiny \begin{array}{c} 2 \end{array}}$ is not $\tau$-rigid and thus $\Fac({\tiny \begin{array}{c} 1 \end{array}} \oplus {\tiny \begin{array}{c} 2 \end{array}}) = \add \{{\tiny \begin{array}{c} 1 \end{array}} \oplus {\tiny \begin{array}{c} 2 \end{array}}\} $ is not a functorially finite torsion class. However
\[ \Filt(\Fac({\tiny \begin{array}{c} 1 \end{array}} \oplus {\tiny \begin{array}{c} 2 \end{array}})) = \add \{ {\tiny \begin{array}{c} 1 \end{array}} \oplus {\tiny \begin{array}{c} 2 \end{array}} \oplus \tiny{\begin{array}{cc} 1 \\2 \end{array}} \} = \Fac({\tiny \begin{array}{c} 2 \end{array}} \oplus \tiny{\begin{array}{cc} 1 \\2 \end{array}}).\]
\end{exmp}

In this example, we also showcased a $\tau$-rigid module that is not support $\tau$-tilting, namely $\tiny{\begin{array}{cc} 1 \\2 \end{array}}$. In general, the following two torsion pairs associated to a $\tau$-rigid pair $(T,P)$ are of particular interest:
\[ (\Fac T, T^\perp) \quad \text{and} \quad ({^\perp \tau T \cap P^\perp}, \Sub(\tau T \oplus \nu P)),\]

where $\nu$ is the Nakayama functor. Adachi Iyama and Reiten \cite{AIR2014} showed the following relation between these two torsion-classes.
\begin{prop} \cite[Corollary 2.13]{AIR2014}
    Let $(T,P)$ be $\tau$-rigid, then $\Fac T \subseteq {^\perp \tau T \cap P^\perp}$. Moreover the equality holds if $(T,P)$ is $\tau$-tilting.
\end{prop}
\begin{proof}
We only show here the inclusion $\Fac T \subseteq {^\perp \tau T \cap P^\perp}$.
Let $M \in \Fac T$, then there exists an epimorphism $T^r \xrightarrow[]{p} M \to 0$ for some $r \in \N$. Assume there exists a nonzero $f \in \Hom(M, \tau T)$, then $fp: T \to \tau T$ is a nonzero map, contradicting the fact that $T$ is $\tau$-rigid. On the other hand, if $g \in \Hom(P,M)$ is nonzero, then the projectivity of $P$ implies the existence of a map $h: P \to T^r$ such that the following diagram commutes
\[ \begin{tikzcd} & P \arrow[ld, "\exists h", dashed, swap] \arrow[d,"g"] \\ T^r \arrow[r, "p"] & M \arrow[r] & 0. \end{tikzcd}\]
Again, since $(T,P)$ is $\tau$-rigid $h = 0$ and thus $g= 0$. In conclusion $M \in {}^\perp \tau T \cap P^\perp$. 
\end{proof}

Since $\Fac T$ is contained in ${^\perp \tau T \cap P^\perp}$ a natural question is to ask how many functorially finite torsion classes $\Tcal$ there are such that $\Fac T \subset \Tcal \subset {^\perp \tau T \cap P^\perp}$. In other words, how many completions are there of a  $\tau$-rigid $(T,P)$ into a $\tau$-tilting pair? We call a $\tau$-tilting pair \textit{almost $\tau$-tilting} if $|T|+|P| = |A|-1$. The following property of almost $\tau$-tilting pairs is one of the motivating factors that lead to the introduction of $\tau$-tilting theory. 

\begin{thm}\label{theorem:twocompletions} \cite[Theorem 2.18]{AIR2014}
    Let $(T,P)$ be an almost $\tau$-tilting pair. Then there are exactly two completions of $(T,P)$ into a $\tau$-tilting pair.
\end{thm}

The theorem above allows us to define \textit{mutation} of $\tau$-tilting pairs, where we delete one indecomposable direct summand of $(T,P)$ and complete it to the unique different $\tau$-tilting pair. The following definition formalises this and distinguishes between the two completions of an almost $\tau$-tilting pair.

\begin{defn} \label{defn:taumutation}
    Let $(T,P)$ be an almost $\tau$-tilting pair, and let $(T',P')$ and $(T'',P'')$ be the two completions of $(T,P)$ into  a $\tau$-tilting pair. Then we say that $(T',P')$ and $(T'',P'')$ are \textit{mutations} of each other. And we say that $(T'',P'')$ is the \textit{left mutation} of $(T',P')$ if $\Fac T'' = \Fac T$ and $\Fac T' = {^\perp \tau T \cap P^\perp}$.
\end{defn}

Jasso \cite{Jasso2015} generalised \cref{theorem:twocompletions} to any number of indecomposable direct summands in the following way. This process is usually called \textit{$\tau$-tilting reduction}.

\begin{thm} \cite[Theorem 1.1]{Jasso2015}
For every $\tau$-rigid pair $(T,P)$ there exists an algebra $B_{(T,P)}$ (the $\tau$-tilting reduction of $A$ by $(T,P)$) and a bijection
\[ \{ \tau\text{-tilting pairs in $\mods B_{(T,P)}$} \} \longleftrightarrow \{ \text{completions of $(T,P)$ to a $\tau$-tilting pair}\}.\]
\end{thm}

The module category of this $\tau$-tilting reduction may be expressed in the following way.

\begin{thm} \cite[Theorem 1.4]{Jasso2015} \label{thm:jassotauperp}
    The category $\mods B_{(T,P)}$ is equivalent to the $\tau$-perpendicular category ${T^\perp\cap{}^\perp \tau T \cap P^\perp}$ of $(T,P)$ and
    \[ | B_{(T,P)} | = |A|-|T|-|P|.\]
\end{thm}

In the following example we illustrate how to use the fact that there are two completions of any almost $\tau$-tilting pair into a $\tau$-tilting pair to obtain all $\tau$-tilting pairs of the algebra. In other words we use mutation starting at the $\tau$-tilting pair $(\bigoplus_{i=1}^n P(i),0)$ to obtain all other $\tau$-tilting pairs. We note that this is not possible for algebras with infinitely many $\tau$-tilting pairs.

\begin{exmp} \label{exmp:as3tiltingpairs}
    Consider the path algebra $A$ over $Q = \begin{tikzcd} 1 \arrow[r] & 2 \arrow[r] & 3 \arrow[ll,bend right] \end{tikzcd}$ modulo the square of the arrow ideal. To find all $\tau$-tilting pairs we start off at the $\tau$-tilting pair where the $\tau$-rigid part is the direct sum of the projectives and then mutate at each of the indecomposable direct summands which means deleting that summand and completing the remaining $\tau$-rigid pair to a $\tau$-tilting pair in the unique way different to the original pair according to \cref{theorem:twocompletions}. We get the following 14 $\tau$-tilting pairs in \cref{fig:mutationdiagram}.
    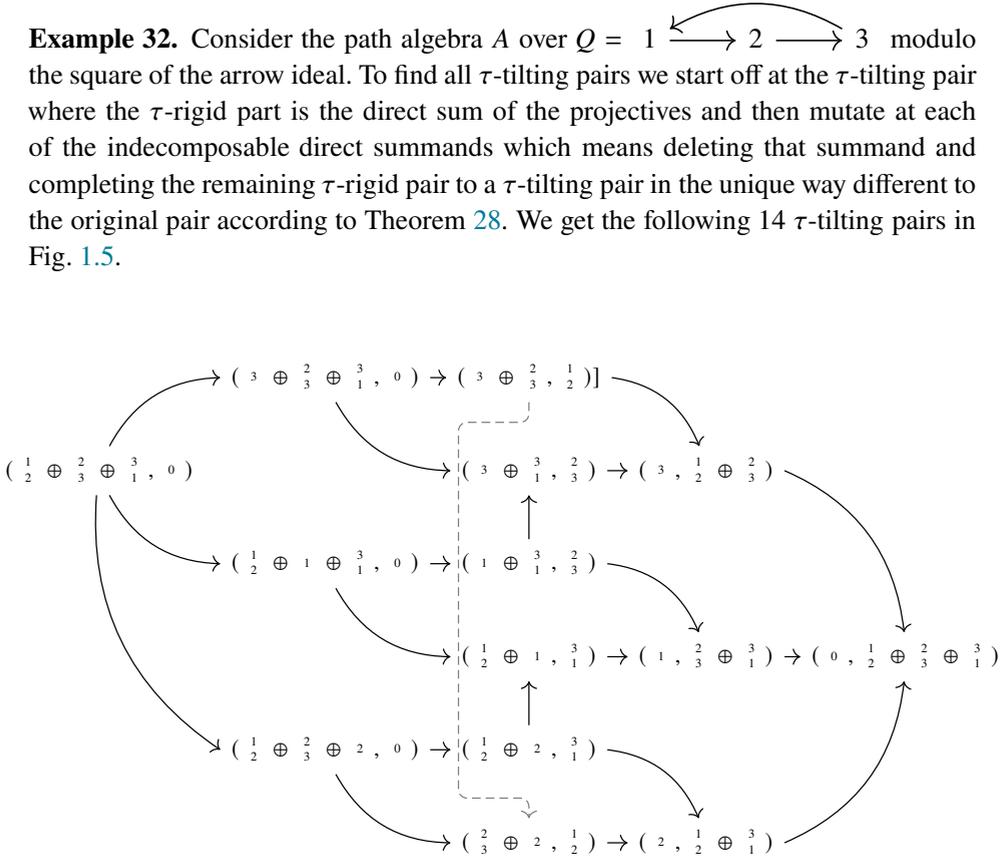
\begin{figure}[ht!] 
    \[ 
    \adjustbox{scale=0.8,center}{\begin{tikzcd}[ampersand replacement=\&,column sep = 0.75em, row sep=2em]
    \& ({\tiny \begin{array}{c} 3 \end{array}} \oplus {\tiny\begin{array}{cc} 2 \\ 3 \end{array}} \oplus {\tiny\begin{array}{cc} 3 \\1 \end{array}},{\tiny\begin{array}{cc} 0 \end{array}}) \arrow[r,start anchor = east, end anchor=west] \arrow[rd, bend right, end anchor=west] \&  ({\tiny \begin{array}{c} 3 \end{array}} \oplus {\tiny\begin{array}{cc} 2\\3 \end{array}}, {\tiny\begin{array}{cc} 1 \\ 2 \end{array}}) 
    \arrow[rd, bend left,start anchor = east] \arrow[ddddd, dashed, gray,
    rounded corners, to path={ -- ([yshift=-2ex]\tikztostart.south) -- ([yshift=-2ex,xshift=-7ex]\tikztostart.south)
    -- ([yshift=2ex,xshift=-7ex]\tikztotarget.north) -- ([yshift=2ex]\tikztotarget.north)
    -- (\tikztotarget)}] ] \\
    ({\tiny\begin{array}{cc} 1 \\ 2 \end{array}} \oplus {\tiny\begin{array}{cc} 2 \\ 3 \end{array}} \oplus {\tiny\begin{array}{cc} 3 \\1 \end{array}},{\tiny\begin{array}{cc} 0 \end{array}}) \arrow[ru,bend left, end anchor=west]\arrow[rd,bend right,end anchor=west] \arrow[rddd,bend right,, end anchor=west]\& \& ({\tiny \begin{array}{c} 3 \end{array}} \oplus {\tiny\begin{array}{cc} 3 \\1 \end{array}}, {\tiny\begin{array}{cc} 2 \\ 3 \end{array}}) \arrow[r, crossing over, start anchor = east, end anchor=west] \& ({\tiny \begin{array}{c} 3 \end{array}}, {\tiny\begin{array}{cc} 1 \\ 2 \end{array}}  \oplus {\tiny\begin{array}{cc} 2 \\ 3 \end{array}}) \arrow[rdd, bend left,start anchor = east] \\
     \&({\tiny\begin{array}{cc} 1 \\ 2 \end{array}} \oplus {\tiny \begin{array}{c} 1 \end{array}} \oplus {\tiny\begin{array}{cc} 3 \\1 \end{array}},{\tiny\begin{array}{cc} 0 \end{array}}) \arrow[r,start anchor = east, end anchor=west] \arrow[rd, bend right, end anchor=west] \& ({\tiny \begin{array}{c} 1 \end{array}} \oplus {\tiny\begin{array}{cc} 3 \\ 1 \end{array}}, {\tiny\begin{array}{cc} 2 \\ 3 \end{array}}) \arrow[rd, bend left, crossing over,start anchor = east] \arrow[u,end anchor = south, start anchor=north]\\
    \& \& ({\tiny\begin{array}{cc} 1 \\ 2 \end{array}} \oplus {\tiny \begin{array}{c} 1 \end{array}}, {\tiny\begin{array}{cc} 3 \\ 1 \end{array}}) \arrow[r, crossing over,start anchor = east, end anchor=west] \& ({\tiny\begin{array}{cc} 1 \end{array}},{\tiny\begin{array}{cc} 2 \\ 3 \end{array}} \oplus {\tiny\begin{array}{cc} 3 \\1 \end{array}}) \arrow[r, start anchor = east, end anchor=west] \& ({\tiny\begin{array}{cc} 0 \end{array}},{\tiny\begin{array}{cc} 1 \\ 2 \end{array}} \oplus {\tiny\begin{array}{cc} 2 \\ 3 \end{array}} \oplus {\tiny\begin{array}{cc} 3 \\1 \end{array}})\\
    \& ({\tiny\begin{array}{cc} 1 \\ 2 \end{array}} \oplus {\tiny\begin{array}{cc} 2 \\ 3 \end{array}} \oplus {\tiny \begin{array}{c} 2 \end{array}}, {\tiny\begin{array}{cc} 0 \end{array}}) \arrow[r,start anchor = east, end anchor=west] \arrow[rd, bend right, end anchor=west] \& ({\tiny\begin{array}{cc} 1 \\ 2 \end{array}} \oplus {\tiny \begin{array}{c} 2 \end{array}}, {\tiny\begin{array}{cc} 3 \\ 1 \end{array}}) \arrow[rd, bend left,crossing over,start anchor = east]  \arrow[u, end anchor = south, start anchor=north] \\
    \& \& ({\tiny\begin{array}{cc} 2 \\3 \end{array}} \oplus {\tiny \begin{array}{c} 2 \end{array}}, {\tiny\begin{array}{cc} 1 \\ 2 \end{array}}) \arrow[r] \& ({\tiny \begin{array}{c} 2 \end{array}}, {\tiny\begin{array}{cc} 1 \\ 2 \end{array}} \oplus {\tiny\begin{array}{cc} 3 \\ 1 \end{array}}) \arrow[ruu, bend right,swap, start anchor = east]
    \end{tikzcd}}
    \]
    \caption{Mutation of $\tau$-tilting pairs} \label{fig:mutationdiagram}
    \label{fig:A3mutation}
    \end{figure}
\end{exmp}

In fact, we may define a partial ordering on $\tau$-tilting pairs corresponding to \cref{defn:taumutation}, where we say that $(T,P) < (T',P')$ if $(T,P)$ a left mutation of $(T',P')$ in other words, if $\Fac T \subset \Fac T'$. The arrows in the mutation graph, \cref{fig:mutationdiagram}, are such that there is an arrow from $(T',P')$ to $(T,P)$ when $(T,P) < (T',P')$.

\begin{exmp}\label{exmp:A2Hasse}
Let us consider again the path algebra $A$ over the quiver $Q: \begin{tikzcd} 1 \arrow[r] & 2 \end{tikzcd}$, we demonstrated the bijection between $\tau$-tilting pairs and torsion classes in \cref{exmp:A2}. We have found the following two chains of inclusions
\[ \add \{ {\tiny\begin{array}{cc} 0 \end{array}} \} \subset \add\{ {\tiny\begin{array}{cc}  2 \end{array}} \} \subset A,\]
\[ \add \{{\tiny\begin{array}{cc} 0 \end{array}}\} \subset \add \{ {\tiny\begin{array}{cc} 1 \end{array}} \} \subset \add \{ 1 \oplus {\tiny\begin{array}{cc} 1 \\ 2 \end{array}} \} \subset A. \]
The induced partial order on the corresponding $\tau$-tilting pairs can be visualised in a Hasse quiver as follows:
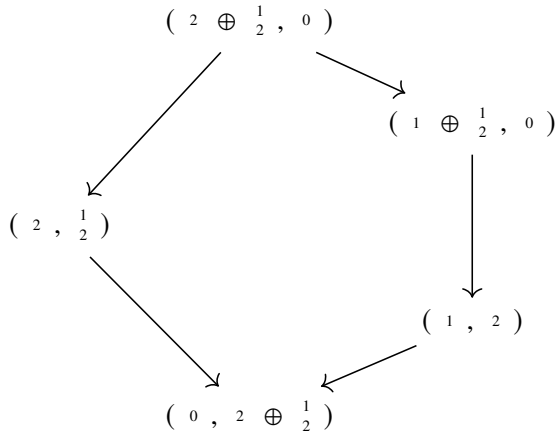
\[
\begin{tikzcd}[ampersand replacement=\&,column sep = 1em, row sep=1.5em]
\& ( {\tiny\begin{array}{cc} 2 \end{array}} \oplus {\tiny\begin{array}{cc} 1\\2 \end{array}}, {\tiny\begin{array}{cc} 0 \end{array}}) \arrow[ldd] \arrow[rd] \\
\& \& ({\tiny\begin{array}{cc} 1 \end{array}} \oplus {\tiny\begin{array}{cc} 1\\2 \end{array}}, {\tiny\begin{array}{cc} 0 \end{array}}) \arrow[dd] \\
({\tiny\begin{array}{cc} 2 \end{array}}, {\tiny\begin{array}{cc} 1 \\2 \end{array}}) \arrow[rdd]  \\
\& \& ({\tiny\begin{array}{cc} 1 \end{array}}, {\tiny\begin{array}{cc} 2 \end{array}}) \arrow[ld] \\
\& ({\tiny\begin{array}{cc} 0 \end{array}}, {\tiny\begin{array}{cc} 2 \end{array}} \oplus {\tiny\begin{array}{cc} 1 \\ 2 \end{array}}) 
\end{tikzcd}
\]
\end{exmp}

\subsection{g-vectors}

Beside the dimension vector $\bdim M$ of an $A$-module $M$, we can also associate another integer vector to the module, called the \textit{$g$-vector} of $M$. The $g$-vector was first introduced by Fomin and Zelevinsky \cite{FominZelevinsky2007} in the context of cluster algebras. Later, it was shown that the $g$-vector is encoded in the projective presentation of $\tau$-rigid modules. It turns out that $g$-vectors of $\tau$-rigid pairs have a lot of underlying structure to them.

\begin{defn}
    Let $M \in \mods A$ be an $A$-module, and $P_{-1} \to P_0 \to M \to 0$ be a minimal projective presentation of $M$ where $P_0 = \bigoplus_{i=1}^n P(i)^{a_i}$ and $P_{-1} = \bigoplus_{i=1}^n P(i)^{b_i}$, then the \textit{$g$-vector} of $M$ is
    \[ g^M \coloneqq (a_1 - b_1, a_2-b_2, \dots, a_n-b_n).\]
\end{defn}

We call $g^T - g^P$ the $g$-vector of the $\tau$-rigid pair $(T,P)$. The following result is fundamental in the study of $g$-vectors and its idea can already be found in \cite{AR1985}, but was first shown by Dehy and Keller \cite{DehyKeller2010} by using a geometric approach in the context of 2-Calabi-Yau categories over an algebraically closed field. Later, Adachi, Iyama and Reiten \cite{AIR2014} adapted it to the language of $\tau$-tilting theory and Demonet, Iyama and Jasso \cite{DIJ2019} extended it to an arbitrary field. 

\begin{thm}\cite[Theorem 2.3]{DehyKeller2010} \cite[Theorem 6.5]{DIJ2019} \label{thm:isogvecs}
    Let $(T,P)$ and $(T',P')$ be $\tau$-rigid pairs then $g^T - g^P = g^{T'} - g^{P'}$ if and only if $T \cong T'$ and $P \cong P'$.
\end{thm}

It follows trivially that if two $\tau$-rigid modules $T$ and $T'$ share a $g$-vector $g^T = g^{T'}$, then $T \cong T'$. On top of that, the $g$-vectors of $\tau$-tilting pairs exhibit many desirable properties. The first of which is the following. 
\begin{thm} \cite[Theorem 5.1]{AIR2014} \label{thm:gvecbasis}
    If $(T,P)$ is a $\tau$-tilting pair then $\{ g^{T_i},-g^{P_i}\}$ is a basis of $\Z^n$.
\end{thm}

The following theorem was proven by Auslander and Reiten in the 1980s, but with the development of $\tau$-tilting theory has suddenly become a key result in this field.
Here recall from \cref{sec:prelim}
that \(\langle v, w \rangle = v^T D_A w\).

\begin{thm} \cite[Theorem 1.4]{AR1985} \label{thm:gvectorsplit}
    Let $M,N \in \mods A$, then
    \[ \langle g^M, \bdim N \rangle = \dim_K \Hom_A(M,N) - \dim_K \Hom_A(N, \tau M).\]
\end{thm}
\begin{proof}
    Let $P_{-1} \to P_0 \to M$ be a minimal projective presentation of $M$. 
    Let us write $P_0 = \bigoplus_{i=1}^n P(i)^{a_i}$ and $P_{-1} = \bigoplus_{i=1}^n P(i)^{b_i}$. 
    Since \(\langle -, - \rangle\) is bilinear we have that
     \[ \langle g^M, \bdim N \rangle= \sum_{i=1}^n (a_i - b_i) \dim_K \Hom_A (P(i),N).\]
    We know that $\dim_K \Hom_A(P(i),N) = \dim_K \End_A(S(i)) \dim_K N_i$, where $N_i$ is the vector space at vertex $i$ of the representation $N$. Furthermore $g^{P(i)} = \boldsymbol{e}_i$, where $\boldsymbol{e}_i$ has nonzero entry equal to 1 only in the $i$-th position. Therefore,
    \[ \langle g^{P(i)}, \bdim N \rangle= \langle \boldsymbol{e}_i, \bdim N \rangle = \dim_K \End_A(S(i)) \dim_K N_i = \dim_K \Hom_A (P(i),N).\]
\end{proof}

Let us now study the arrangement of $g$-vectors in $\R^n$. In particular, we consider the cones spanned by $g$-vectors of $\tau$-rigid objects in the following manner.

\begin{defn}
    Let $(T,P)$ be a $\tau$-rigid pair so $T= \bigoplus_{i=1}^k T_i$ and $P = \bigoplus_{i=k+1}^t P_i$ for some $t \leq n$, then we define the polyhedral cone $\Ccal_{(T,P)}$ to be given by
    \[ \Ccal_{(T,P)} = \left\{ \sum_{i=1}^k \alpha_i g^{T_i} -  \sum_{j=k+1}^t \alpha_j g^{P_j}: \alpha_i \geq 0 \text{ for all } 1 \leq i \leq t \right\} \subseteq \R^n.\]
    The interior cone $\Ccal_{(T,P)}^o$ is defined as
    \[ \Ccal_{(T,P)}^o = \left\{ \sum_{i=1}^k \alpha_i g^{T_i} -  \sum_{j=k+1}^t \alpha_j g^{P_j}: \alpha_i >0  \text{ for all } 1 \leq i \leq t \right\} \subseteq \R^n.\]
\end{defn}

Demonet, Iyama and Jasso \cite{DIJ2019} used \cref{thm:isogvecs} to show that cones of $g$-vectors of $\tau$-rigid pairs form what is called a \textit{polyhedral fan}. 

\begin{thm} \label{lem:coneintersect}\cite[Corollary 6.7]{DIJ2019} Let $(T_1, P_1), (T_2,P_2)$ be two $\tau$-rigid pairs. Let $(T,P)$ be the maximal common direct summand of $(T_1,P_1)$ and $(T_2,P_2)$, then $\Ccal_{(T_1,P_1)} \cap \Ccal_{(T_2,P_2)} = \Ccal_{(T,P)}$.
\end{thm}

As Brüstle, Smith and Treffinger \cite{BST2019} have shown, one can express the category of $v$-semistable $A$-modules locally in terms of $\tau$-rigid pairs in the following way.
\begin{thm} \cite[Proposition 3.13]{BST2019} \label{thm:ssmodsequiv} Let $(T,P)$ be a $\tau$-rigid pair, then 
    \[\mods_v^{ss} A = T^\perp \cap {^\perp \tau T \cap P^\perp} \text{ for all } v \in \Ccal_{(T,P)}^o.\]
\end{thm}

\begin{proof} Let $v \in \Ccal_{(T,P)}^o$ and $(T,P)$ be a $\tau$-rigid pair. We split the proof into showing that one side is contained in the other and vice versa. Throughout we make use of the identity
\begin{align*}
\langle v, \bdim M \rangle &= \sum_{i=1}^k \alpha_i \langle g^{T_i}, \bdim M \rangle - \sum_{k}^t \alpha_j \langle g^{P_j}, \bdim M \rangle \\
&= \sum_{i=1}^k \alpha_i \Hom(T_i,M) - \sum_{i=1}^k \alpha_i \Hom(M,\tau T_i) - \sum_{j=k+1}^t \Hom (P_j,M)
\end{align*}
derived from the definition of the cone corresponding to the $\tau$-rigid pair $(T,P)$ and \cref{thm:gvectorsplit}.
\begin{enumerate}
    \item Let $M \in \mods_v^{ss} A$. Then we know $(\Fac T, T^\perp)$ is a torsion pair i.e. there exists a short exact sequence
    \[ 0 \to \underbrace{tM}_{\in \Fac T} \to M \to \underbrace{fM}_{\in T^\perp} \to 0.\]
    Since $tM \to M$ is an injection and $M$ is $v$-semistable we know that $\langle V, \bdim tM \rangle \leq 0$ and therefore
    \begin{align*}
        0 &\geq \langle v, \bdim tM \rangle \\
        & = \sum_{i=1}^k \alpha_i \dim \Hom (T_i, tM) - \underbrace{\sum_{i=1}^k \alpha_i \dim \Hom (tM,\tau T_i)}_{=\;0} \\
        &\phantom{= \sum_{i=1}^k \alpha_i \dim \Hom (T_i, tM)} - \underbrace{\sum_{j=k+1}^t \alpha_j \dim \Hom( P_j, tM)}_{=\;0} \\
        &= \sum_{i=1}^k \alpha_i \dim \Hom(T_i, tM)
    \end{align*}
    where the last two expressions equal zero because $(T,P)$ is $\tau$-rigid. In particular $\Hom(T,\tau T)=0$ and $\Hom(P,T)=0$. Furthermore, since $\alpha_i > 0$ and $\bdim \Hom_A{(T_i,tM)}\geq 0$ this implies $\Hom(T,tM)=0$. Therefore $tM \cong 0$ and thus $\Hom(T,M)=0$ which means $M \in T^{\perp}$.  \\

    From this it immediately follows that the semistable object $M$ satisfies
    \[ 0 = \langle v, \bdim M \rangle = - \sum_{i=1}^k \alpha_i \dim \Hom(M, \tau T_i) - \sum_{j=k+1}^t \alpha_j \dim \Hom (P_j,M). \]
    So that again, since $\alpha_i > 0$ for all $1 \leq i \leq t$ we get $\Hom(M,\tau T) = 0$ and $\Hom(P, M)=0$. In other words $M \in {^\perp} \tau T \cap P^\perp$ and thus combining with the we obtain 
    \[M \in {T^\perp \cap {^\perp \tau T} \cap P^\perp}. \]
    \item Let $M \in {T^\perp \cap {^\perp \tau T} \cap P^\perp}$, then we immediately get
    \begin{align*}
    \langle v, \bdim M \rangle & =  \sum_{i=1}^k \alpha_i \dim \Hom (T_i, tM) \\
    &- \sum_{i=1}^k \alpha_i \dim \Hom (M,\tau T_i) \\
    &- \sum_{j=k+1}^t \alpha_j \dim \Hom( P_j, M) =0 
    \end{align*}
    since each of the individual terms lie in the respective $\Hom$-perpendicular classes. Consider now $L \hookrightarrow M \in {T^\perp \cap {^\perp \tau T} \cap P^\perp} \subset T^\perp$. Since $L$ is a submodule of $M$ it follows that $L \in T^\perp$. Therefore
    \begin{align*}
        \langle v, \bdim L \rangle & =  \underbrace{\sum_{i=1}^k \alpha_i \dim \Hom (T_i, L)}_{=0} \\
        &- \sum_{i=1}^k \alpha_i \dim \Hom (L,\tau T_i) \\
        &- \sum_{j=k+1}^t \alpha_j \dim \Hom( P_j, L) \leq 0 
    \end{align*}
    as required. Thus $M \in \mods_{v}^{ss}A$. 
\end{enumerate}
Combining the two halves we obtain $\mods_v^{ss} A = {T^\perp \cap {^\perp \tau T} \cap P^\perp}$. 
\end{proof}

\begin{exmp} \label{exmp:A2gvecs}
    Take $Q= \begin{tikzcd} 1 \arrow[r] & 2 \end{tikzcd}$ then the Auslander-Reiten quiver is the given above in \cref{exmp:A2}. The projective presentations of the three modules ${\tiny \begin{array}{c} 2 \end{array}} = P(2), {\tiny{\begin{array}{cc} 1 \\2 \end{array}}} = P(1), {\tiny \begin{array}{c} 1 \end{array}} = I(1)$ are as follows
    \[ {\tiny\begin{array}{cc} 0 \end{array}} \to \underbrace{P(1)^0 \oplus P(2)^1}_{=\;2} \to {\tiny \begin{array}{c} 2 \end{array}} \to {\tiny\begin{array}{cc} 0 \end{array}} \quad \Longrightarrow \quad g^{2} = (0-0,1-0)=(0,1),\]
    \[ {\tiny\begin{array}{cc} 0 \end{array}} \to \underbrace{P(1)^1 \oplus P(2)^0}_{=\tiny{\begin{array}{cc} 1 \\2 \end{array}}} \to {\tiny{\begin{array}{cc} 1 \\2 \end{array}}} \to {\tiny\begin{array}{cc} 0 \end{array}} \quad \Longrightarrow \quad g^{\tiny{\begin{array}{cc} 1 \\2 \end{array}}} = (1-0,0-0)=(1,0),\]
    \[ P(1)^0 \oplus P(2)^1 \to P(1)^1 \oplus P(2)^0 \to {\tiny \begin{array}{c} 1 \end{array}} \to {\tiny\begin{array}{cc} 0 \end{array}} \quad \Longrightarrow \quad g^1 = (1-0,0-1)=(1,-1).\]
\end{exmp}
The set of $g$-vectors of the indecomposable $\tau$-rigid pairs is therefore given by
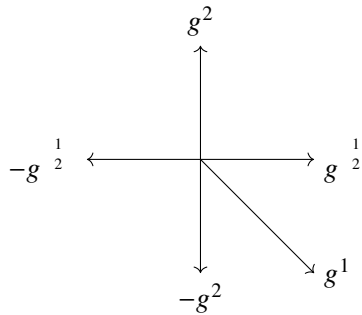
\[
\begin{tikzpicture}
  \draw[->] (0, 0) -- (1.5, 0) node[right] {$g^{\tiny{\begin{array}{cc} 1 \\2 \end{array}}}$};
  \draw[->] (0, 0) -- (0, 1.5) node[above] {$g^2$};
  \draw[->] (0, 0) -- (-1.5, 0) node[left] {$-g^{\tiny{\begin{array}{cc} 1 \\2 \end{array}}}$};
  \draw[->] (0, 0) -- (0, -1.5) node[below] {$-g^2$};
  \draw[->] (0,0) -- (1.5,-1.5) node[right] {$g^1$};
\end{tikzpicture}
\]

\section{From $\tau$-tilting theory to wall-and-chamber structures}
Recall, that $K_0(A)$ is the Grothendieck group of an algebra $A$, the free abelian group isomorphic to $\Z^n$ having as basis the set $\{ [S(1)], [S(2)], \dots, [S(n)] \}$ of isomorphism classes of simple right $A$-modules. Recall from \cref{thm:ssmodsequiv} that if $(T,P)$ is a $\tau$-rigid pair then we have an equivalence of categories
\[ \mods_v^{ss} A = T^\perp \cap {^\perp \tau T \cap P^\perp}\]
for every $v \in \Ccal_{(T,P)}^o$. The following result tells us how many $v$-stable modules there are when $v$ is in the interior cone of a $\tau$-rigid pair $(T,P)$.

\begin{prop} \cite[Theorem 3.14]{BST2019} \label{thm:numbervssmods} Let $(T,P)$ be a $\tau$-rigid pair and let $v \in \Ccal_{(T,P)}^o$. Then there are exactly
\[\rank \left(K_0 \left(T^\perp \cap {^\perp \tau T} \cap P^\perp \right) \right) = |A| - |T| - |P|\]
$v$-stable modules. In particular, there are no nonzero $v$-stable modules when $(T,P)$ is $\tau$-tilting.
\end{prop}
\begin{proof}
This is a combination of \cref{thm:jassotauperp} and \cref{thm:ssmodsequiv}.
\end{proof}

The following result establishes a bijection between chambers and $\tau$-tilting pairs and was first shown by Brüstle, Smith and Treffinger \cite{BST2019} who proved that every $\tau$-tilting pair gives rise to a unique chamber. Later, Asai \cite{Asai2019} showed the converse.

\begin{thm} \cite[Proposition 3.15]{BST2019}\cite[Theorem 3.17]{Asai2019}
\label{thm:ttilt-chamber-bij}
    Let $(T,P)$ be a $\tau$-tilting pair. Then $\Ccal_{(T,P)}^o$ is a chamber, that is, a connected open component of $\Rfrak \coloneqq \R^n \setminus \overline{\bigcup_{\substack{M \in \mods A\\0 \neq M}} \Dcal(M)} $. Moreover every chamber arises this way.
\end{thm}

\begin{proof} Take $v_1 \in \Ccal_{(T,P)}^o$. Since $(T,P)$ is a $\tau$-tilting pair, we know $\mods_{v_1}^{ss} A = \{0\}$ from \cref{thm:numbervssmods}. In other words, $v_1$ cannot be in the stability space of any module, thus $v_1$ belongs to a chamber $\Cfrak$. Moreover, every vector of $\Ccal_{(T,P)}^o$ belongs to the same chamber, since $\Ccal_{(T,P)}^o$ is connected. Hence $\Ccal_{(T,P)}^o \subset \Cfrak$.\\

On the other hand, take a vector $v_2$ in the boundary i.e.
\[v_2 \in \Ccal_{(T,P)} \setminus \Ccal_{(T,P)}^o = \left\{ \sum_{i=1}^k \beta_i g^{T_i} -  \sum_{j=k+1}^t \beta_j g^{P_j}: \beta_i 
= 0  \text{ for some } 1 \leq i \leq t \right\}. \]
By definition there exists at least one index $i \in \{1,\dots,n\}$ such that $\beta_i = 0$. Therefore there is smaller-dimensional cone corresponding to a $\tau$-rigid pair $(T',P')$. From \cref{thm:numbervssmods} it follows that
\[ \mods_{v_2}^{ss}A = (T')^\perp \cap {^\perp \tau T'} \cap (P')^\perp \neq \{ 0 \}. \]
In other words, $v_2$ is in the stability space of some module, and thus is in a wall, hence no vectors outside of $\Ccal_{(T,P)}^o$ are in the chamber since it is connected and therefore $\Ccal_{(T,P)}^o = \Cfrak$. The proof of the converse can be found in \cite{Asai2019}.
\end{proof}

\underline{Remark.} Let $(T,P)$ be an almost $\tau$-tilting pair and let $(T',P')$ and $(T'',P'')$ be the two completions of $(T,P)$ into a $\tau$-tilting pair.
\begin{enumerate} 
    \item Then $\codim (\Ccal_{(T,P)}^o) = |A| - |T] -|P|=1$ i.e. the cone corresponding to the almost $\tau$-tilting pair is a wall.
    \item Like in \cref{lem:coneintersect} we have $\Ccal_{(T',P')} \cap \Ccal_{(T'',P'')} = \Ccal_{(T,P)}$. In other words the intersection of the chambers corresponding to the two completions is the wall separating them, given by the cone of the original $\tau$-rigid pair.
\end{enumerate}

The last remark illustrates that in the wall-and-chamber structure a mutation of a $\tau$-tilting pair corresponds to crossing a wall from the chamber corresponding to the pair to the chamber corresponding to the mutation of it. \\

In other words, \cref{thm:numbervssmods} says that each wall in the wall-and-chamber structure corresponds to a unique $v$-stable module where $v$ is in the wall. So if mutation corresponds to crossing a wall we can associate to any mutation of a $\tau$-tilting pair, the brick $B \in \mods_v^{ss} A$ whose stability space is the wall we cross. This is called the ``brick labelling for functorially finite torsion classes'' which we demonstrate below.

\begin{exmp}
    Consider the wall-and-chamber structure of the path algebra $A$ over the quiver $Q= \begin{tikzcd} 1 \arrow[r] & 2 \end{tikzcd}$ calculated in \cref{exmp:A2wac} and given by
        
    
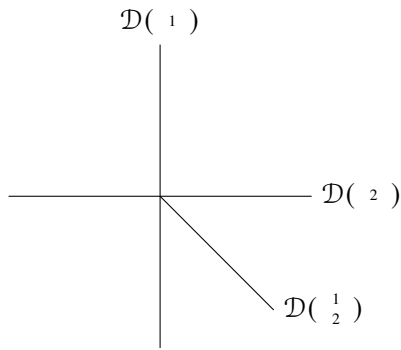
\[
    \begin{tikzpicture}
      \draw[-] (-2, 0) -- (2, 0) node[right] {$\Dcal({\tiny \begin{array}{c} 2 \end{array}})$};
      \draw[-] (0, -2) -- (0, 2) node[above] {$\Dcal({\tiny \begin{array}{c} 1 \end{array}})$};
      \draw[-] (0,0) -- (1.5,-1.5) node[right] {$\Dcal(\tiny{\begin{array}{cc} 1 \\2 \end{array}}\normalsize)$};
    \end{tikzpicture}
    \]
    We have found all the $\tau$-tilting pairs in \cref{exmp:A2Hasse}. In \cref{exmp:A2gvecs} we have found the following g-vectors for the indecomposable modules of $A$
    \[  g^2 =(0,1), \qquad g^{\tiny{\begin{array}{cc} 1 \\2 \end{array}}} = (1,0), \qquad g^1 = (1,-1).\]
    A quick calculation gives us the 5 chambers
    \[ \Ccal_{({\tiny \begin{array}{c} 2 \end{array}} \oplus {\tiny{\begin{array}{cc} 1 \\2 \end{array}}}, {\tiny\begin{array}{cc} 0 \end{array}})}^o = \left\{ \begin{pmatrix} \alpha_1 \\ \alpha_2 \end{pmatrix} : \alpha_i > 0 \right\}, \quad  \Ccal_{({\tiny \begin{array}{c} 1 \end{array}} \oplus {\tiny{\begin{array}{cc} 1 \\2 \end{array}}}, {\tiny\begin{array}{cc} 0 \end{array}})}^o = \left\{ \begin{pmatrix} \alpha_1 + \alpha_2\\ -\alpha_1 \end{pmatrix} : \alpha_i > 0 \right\},\]
    \[ \Ccal_{({\tiny \begin{array}{c} 2 \end{array}}, {\tiny{\begin{array}{cc} 1 \\2 \end{array}}})}^o = \left\{ \begin{pmatrix} -\alpha_2 \\ \alpha_1 \end{pmatrix} : \alpha_i > 0 \right\}, \quad  \Ccal_{({\tiny \begin{array}{c} 1 \end{array}},{\tiny \begin{array}{c} 2 \end{array}})}^o = \left\{ \begin{pmatrix} -\alpha_2 \\ -\alpha_1 \end{pmatrix} : \alpha_i > 0 \right\},\]
    \[ \Ccal_{({\tiny\begin{array}{cc} 0 \end{array}} ,{\tiny \begin{array}{c} 2 \end{array}} \oplus {\tiny{\begin{array}{cc} 1 \\2 \end{array}}})}^o = \left\{ \begin{pmatrix} \alpha_1 \\ -\alpha_1- \alpha_2 \end{pmatrix} : \alpha_i > 0 \right\}\]
    which correspond to the following picture
    
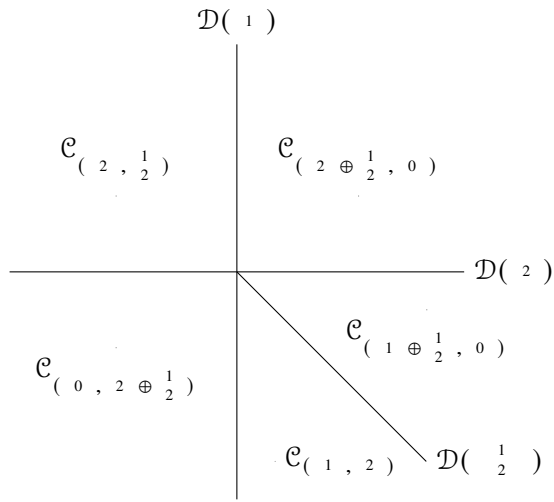
\[ \begin{tikzpicture}
      \draw[-] (-3, 0) -- (3, 0) node[right] {$\Dcal({\tiny \begin{array}{c} 2 \end{array}})$};
      \draw[-] (0, -3) -- (0, 3) node[above] {$\Dcal({\tiny \begin{array}{c} 1 \end{array}})$};
      \draw[-] (0,0) -- (2.5,-2.5) node[right] {$\Dcal(\tiny{\begin{array}{cc} {\tiny \begin{array}{c} 1 \end{array}} \\ {\tiny \begin{array}{c} 2 \end{array}} \end{array}}\normalsize)$};
      \filldraw[black] (1.6,1) circle (0pt) node[above]{$\Ccal_{({\tiny \begin{array}{c} 2 \end{array}} \oplus {\tiny{\begin{array}{cc} 1 \\2 \end{array}}}, {\tiny\begin{array}{cc} 0 \end{array}})}$};
      \filldraw[black] (-1.6,1) circle (0pt) node[above]{$\Ccal_{({\tiny \begin{array}{c} 2 \end{array}} , {\tiny{\begin{array}{cc} 1 \\2 \end{array}}})}$};
      \filldraw[black] (-1.6,-1) circle (0pt) node[below]{$\Ccal_{({\tiny\begin{array}{cc} 0 \end{array}} ,{\tiny \begin{array}{c} 2 \end{array}} \oplus {\tiny{\begin{array}{cc} 1 \\2 \end{array}}})}$};
      \filldraw[black] (0.5,-2.5) circle (0pt) node[right]{$\Ccal_{({\tiny \begin{array}{c} 1 \end{array}},{\tiny \begin{array}{c} 2 \end{array}})}$};
      \filldraw[black] (2.5,-0.5) circle (0pt) node[below]{$\Ccal_{({\tiny \begin{array}{c} 1 \end{array}} \oplus {\tiny{\begin{array}{cc} 1 \\2 \end{array}}}, {\tiny\begin{array}{cc} 0 \end{array}})}$};
    \end{tikzpicture} \]
    This way we obtain the a labelling for arrows in the following mutation diagram such that each arrow indicates the wall we cross when mutating.
    
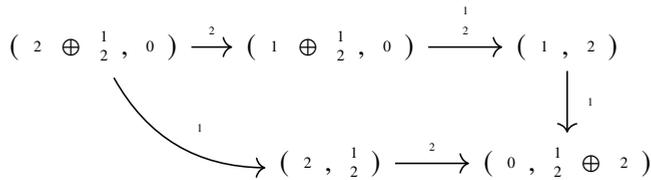
\[
    \begin{tikzcd}[ampersand replacement=\&,column sep = 1.5em, row sep=2em]
        ({\tiny \begin{array}{c} 2 \end{array}} \oplus {\tiny{\begin{array}{cc} 1 \\2 \end{array}}}, {\tiny\begin{array}{cc} 0 \end{array}}) \arrow[r,"\tiny{\begin{array}{cc} 2 \end{array}}"] \arrow[rd, bend right, "\tiny{\begin{array}{cc} 1 \end{array}}"] \& ({\tiny \begin{array}{c} 1 \end{array}} \oplus {\tiny{\begin{array}{cc} 1 \\2 \end{array}}}, {\tiny\begin{array}{cc} 0 \end{array}}) \arrow[r,"{\tiny{\begin{array}{cc} 1 \\2 \end{array}}}"] \& ({\tiny \begin{array}{c} 1 \end{array}},{\tiny \begin{array}{c} 2 \end{array}}) \arrow[d,"\tiny{\begin{array}{cc} 1 \end{array}}"] \\
        \& ({\tiny \begin{array}{c} 2 \end{array}}, {\tiny{\begin{array}{cc} 1 \\2 \end{array}}} )\arrow[r,"\tiny{\begin{array}{cc} 2 \end{array}}"] \& ( {\tiny\begin{array}{cc} 0 \end{array}} ,{\tiny{\begin{array}{cc} 1 \\2 \end{array}}} \oplus {\tiny \begin{array}{c} 2 \end{array}})
    \end{tikzcd}
    \]
\end{exmp}

\subsection{$c$-vectors}

In this subsection, to avoid technicalities, we assume that $K$ is algebraically closed. The notion of $c$-vectors was introduced by Fu \cite{Fu2017} and was motivated by the \textit{tropical duality} of cluster algebras. Let us begin by giving the following constructive proof of \cref{thm:gvecbasis} by Treffinger \cite{Treffinger2019}.

\begin{thm}\cite[Theorem 3.5]{Treffinger2019}\label{thm:cvecsigncoh} Let $(T,P)$ be a $\tau$-tilting pair such that $T \cong \bigoplus_{i=1}^k T_i$ and $P \cong \bigoplus_{j=k+1}^t P_j$, then the set of $g$-vectors $\{ g^{T_1}, \dots, g^{T_k}, -g^{P_{k+1}}, \dots , -g^{P_n}\}$ is a basis of $\Z^n$.
\end{thm}
\begin{proof} Let $G_{(T,P)} \coloneqq (g^{T_1} |\dots |g^{T_k}|-g^{P_{k+1}}| \dots | -g^{P_n})$ be the $n \times n$ matrix whose columns are the $g$-vectors of the indecomposable direct summands of $(T,P)$. Set $(T,P)_\ell = \left( \bigoplus_{i \neq \ell} T_i, \bigoplus_{i \neq \ell} P_i\right)$ for $1 \leq \ell \leq n$. From \cref{thm:ssmodsequiv} it follows that for all $\ell \in \{ 1, \dots, n \}$ there exists a unique brick $B_\ell$ which is $v$-semistable for all $v \in \Ccal_{(T,P)_\ell}$ given by
\[ B_\ell \in (\bigoplus_{i \neq \ell} T_i) ^\perp \cap {^\perp( \bigoplus_{i \neq \ell} \tau T_i )} \cap (\bigoplus_{i \neq l} P_i)^\perp.\]

Treffinger showed in \cite[Lemma 3.3]{Treffinger2019} that the collection of these bricks $B_\ell$ is such that
\[ G_{(T,P)}^T \cdot (\bdim B_1 | \bdim B_2 | \dots |\bdim B_n ) = \begin{pmatrix} \delta_1 & 0 & \dots &0 \\ 0 & \delta_2 & & \vdots \\ \vdots & & \ddots & 0 \\ 0 & \dots & 0 & \delta_n \end{pmatrix}\]
where $\delta_i \in \{-1,1\}$ for all $1 \leq i \leq n$. In other words, the product of the $g$-matrix and the matrix of bricks is an invertible diagonal matrix in $\text{Mat}_n(\Z)$. Thus the $g$-vectors form a basis of $\Z^n$.
\end{proof}

\underline{Remark.} The only place in the previous proof where the hypothesis of $K$ being algebraically closed is used, is to show that $\delta_i \in \{-1,1\}$. However, the fact that the multiplication of these two matrices is diagonal with nonzero determinant is true for every finite dimensional algebra. See for instance \cite[Theorem 1.3]{ST2020}. \\

We may now define the $c$-vectors explicitly in the following way.

\begin{defn} \cite{Fu2017}
    Let $(T,P)$ be a $\tau$-tilting pair. We define the $C$\textit{-matrix} of $(T,P)$ to be
    \[ C_{(T,P)} \coloneqq (G_{(T,P)}^{T})^{-1}. \]
    We call the columns of $C_{(T,P)}$ the $c$\textit{-vectors} of $A$. In particular we call column $i$ the $i$-th $c$-vector and denote it by $\bc_i$. The construction by Treffinger \cite{Treffinger2019} above shows that every $c$-vector is of the form $c= \pm \bdim B$ for some brick $B$, which means the $c$-vectors are sign-coherent.
\end{defn}

Let us now define a class of algebras for which we can obtain all $\tau$-tilting pairs by mutation from the $\tau$-tilting pair $(A,0)$.

\begin{defn}
    An algebra is said to be \textit{$\tau$-tilting finite} (also called \text{$g$-finite}) if the number of $\tau$-tilting pairs is finite.
\end{defn}

There are many equivalent conditions for an algebra to be $\tau$-tilting finite, let us state the following two.

\begin{thm} \cite[Theorem 1.4]{DIJ2019}
    An algebra is $\tau$-tilting finite if and only if the number of bricks in $\mods A$ is finite.
\end{thm}

\begin{thm} \cite[Theorem 1.1]{ST2020}
    An algebra is $\tau$-tilting finite if and only if there exists a number $d \in \N$ such that $\dim_K B \leq d$ for every brick $B \in \mods A$.
\end{thm}

Let us conclude by giving some more insight into $c$-vectors and their associated bricks. Asai \cite{Asai2020} introduced semibricks, defined as follows.

\begin{defn}\cite[Definition 1.1]{Asai2020}
A set $\{ B_i: i \in I\}$ is a \textit{semibrick} if $B_i$ is a brick for each $i \in I$ and $\Hom_A(B_i, B_j) = 0$ for $i \neq j$.
\end{defn}

Treffinger \cite{Treffinger2019} showed that in fact the collection of positive (resp. negative) $c$-vectors form a semibrick. The following statement is in the setting of \cref{thm:cvecsigncoh}.

\begin{thm} \cite[Lemma 3.10]{Treffinger2019}
    Let $(T,P)$ be a $\tau$-tilting pair and let $\{ B_i: i= 1, \dots,n \}$ be the bricks associated to the almost $\tau$-rigid pairs $(T,P)_i$. Define
    \[ C_{(T,P)}^+ \coloneqq \{ B_i: c_i= \bdim B_i \}, \quad C_{(T,P)}^- \coloneqq \{ B_i: c_i = - \bdim B_i \}. \]
    Then $C_{(T,P)}^+$ and $C_{(T,P)}^-$ are semibricks. Moreover, $\Filt(\Fac(C_{(T,P)}^+)) = \Fac T$ and $\Filt(\textnormal{Sub}(C_{(T,P)}^-)) = T^\perp$. 
\end{thm}
\begin{proof}
 Let us begin by showing that $C_{(T,P)}^+$ is a semibrick. Take $B_s, B_t \in C_{(T,P)}^+$ such that $s \neq t$. By construction the bricks $B_s$ and $B_t$ are $v$-semistable for $v \in \Ccal_{(T,P)_s}^o$ and $v \in \Ccal_{(T,P)_t}^o$, respectively, and satisfy 
\[ B_s \in \left(\bigoplus_{i \neq s} T_i\right){}^\perp \cap {}^\perp \left(\bigoplus_{i \neq s} \tau T_i\right), \qquad  B_t \in \left(\bigoplus_{i \neq t} T_i\right){}^\perp \cap {}^\perp \left(\bigoplus_{i \neq t} \tau T_i\right). \]
Then \cite[Proposition 3.2]{Treffinger2019} implies that $B_s, B_t \in \Fac T$. Combining this with the above, it follows that there exists an epimorphism $p_t: T_t \to B_t$. Thus every morphism $f \in \Hom(B_t,B_s)$ can be composed with $p_t$ to get a nonzero map $f p_t: T_t \to B_s$. But since $B_s \in (\bigoplus_{i \neq s} M_i)^\perp$ and $s \neq t$ this is a contradiction. Therefore $f = 0$ and $\Hom(B_t, B_s) = 0$ for all $B_t, B_s \in C_{(T,P)}^+$. \\

A dual argument shows that elements $B_s, B_t \in C_{(T,P)}^-$ such that $s \neq t$ satisfy $B_s, B_t \in \Sub(\tau T)$. In particular there is a monomorphism $\iota_s: B_s \to \tau T_s$ Thus any morphism $f \in \Hom(B_t, B_s)$ can be composed with $\iota_s$ to obtain a map $\iota_s f: B_t \to \tau T_s$ which is a contradiction and $f = 0$. For the moreover part, see \cite[Lemma 3.13]{Treffinger2019}.
\end{proof}

\section{A detailed example}
We conclude these notes by illustrating the connection between wall-and-chamber structures and $\tau$-tilting theory on our running example. Consider the path algebra $A$ over $Q= \begin{tikzcd} 1 \arrow[r] & 2 \arrow[r] & 3 \arrow[ll, bend right] \end{tikzcd}$ modulo the square of the arrow ideal. In \cref{exmp:as3tiltingpairs} we found 14 $\tau$-tilting pair and in \cref{exmp:A3loop} we found 14 chambers. This agrees with the prediction of \cref{thm:ttilt-chamber-bij}, that there is a one-to-one correspondence between chambers and $\tau$-tilting pairs. In order to compute which chamber correspond to the cone of which $\tau$-tilting pair, we must first compute all the $g$-vectors. The simple modules have projective resolutions given by
    \[ {\tiny \begin{array}{cc} 2 \\ 3 \end{array}} \to {\tiny \begin{array}{cc} 1 \\2 \end{array}} \to {\tiny \begin{array}{c} 1 \end{array}} \quad \Longrightarrow \quad g^1 = (1,-1,0),\]
    \[ {\tiny \begin{array}{cc} 3 \\ 1 \end{array}} \to {\tiny \begin{array}{cc} 2 \\3 \end{array}} \to {\tiny \begin{array}{c} 2 \end{array}} \quad \Longrightarrow \quad g^2 = (0,1,-1),\]
    \[ {\tiny \begin{array}{cc} 1 \\ 2 \end{array}} \to {\tiny \begin{array}{cc} 3 \\1 \end{array}} \to {\tiny \begin{array}{c} 3 \end{array}}  \quad \Longrightarrow \quad g^3= (-1,0,1).\]
    Whereas the projective modules have trivial projective resolutions. Their $g$-vectors are the following:
    \[ g^{\tiny\begin{array}{cc} 1 \\2 \end{array}} = (1,0,0), \quad g^{\tiny\begin{array}{cc} 2 \\3 \end{array}} = (0,1,0), \quad g^{\tiny\begin{array}{cc} 3 \\1 \end{array}} = (0,0,1), \quad \]
    Since the cone associated to a $\tau$-tilting pair $(T,P)$ is just the space given by $\sum \alpha_i g^{T_i} - \sum \alpha_j g^{P_j}$, for $\alpha_i, \alpha_j > 0$ it follows that the cone $\Ccal_{(T,P)}$ corresponding to $(T,P)$ is the chamber whose ``vertices'' are the $g$-vectors $g^{T_i}$ and $-g^{P_i}$. We obtain \cref{fig:bigwacexmp} which we can see as a dual to the mutation graph of \cref{fig:A3mutation} in a way that mutation corresponds to crossing a wall. The corresponding torsion pairs to each chamber may be found in \cref{table:A3table}
    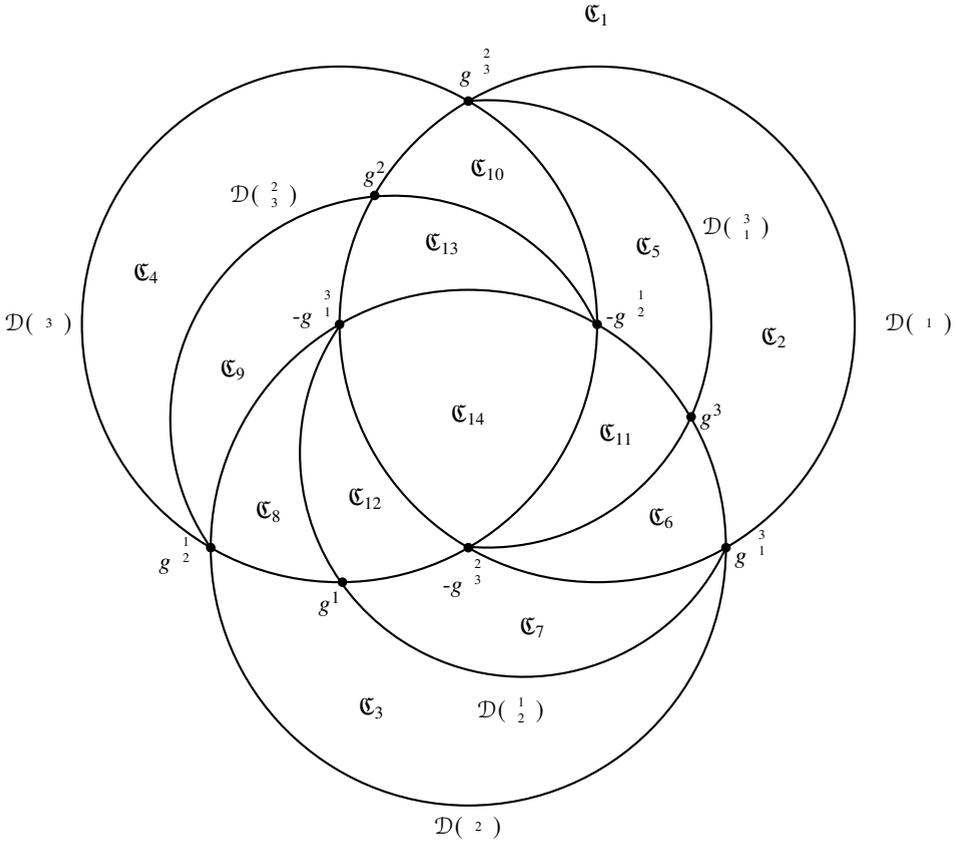
\begin{figure}[ht!]
    \centering
    \begin{tikzpicture}[thick,scale=0.85, every node/.style={transform shape}]
        \draw [black,thick] (-4,-3.4641) arc [
            start angle=-145,
            end angle=-335,
            x radius=3.47cm,
            y radius=3.47cm
        ];
        \draw [black,thick] (4,-3.4641) arc [
            start angle=-25,
            end angle=-215,
            x radius=3.47cm,
            y radius=3.47cm
        ];
        \draw [black,thick] (0,3.4641) arc [
            start angle=-265,
            end angle=-455,
            x radius=3.47cm,
            y radius =3.47cm
        ];
        \draw[color=black,thick] (0,-3.46410161514) circle [radius=4];
        \draw[color=black,thick] (2,0) circle [radius=4];
        \draw[color=black,thick] (-2,0) circle [radius=4];
        \filldraw[black] (-6,0) node[left]{$\Dcal({\tiny \begin{array}{c} 3 \end{array}})$};
        \filldraw[black] (0,-7.5) node[below]{$\Dcal({\tiny \begin{array}{c} 2 \end{array}})$};
        \filldraw[black] (7,0) node[black]{$\Dcal({\tiny \begin{array}{c} 1 \end{array}})$};
        \filldraw[black] (0,-6) node[right] {$\Dcal({\tiny \begin{array}{cc} 1 \\2 \end{array}})$};
        \filldraw[black] (-2.5,2) node[left] {$\Dcal({\tiny \begin{array}{cc} 2 \\3 \end{array}})$};
        \filldraw[black] (3.5,1.5) node[right] {$\Dcal({\tiny \begin{array}{cc} 3 \\1 \end{array}})$};
        \filldraw[black] (4,-3.4641) node[right] {$g^{\tiny \begin{array}{cc} 3 \\1 \end{array}}$};
        \node at (4,-3.4641)[circle,fill,inner sep=1.5pt]{};
        \filldraw[black] (0.2,3.55) node[above] {$g^{\tiny \begin{array}{cc} 2 \\ 3 \end{array}}$};
        \node at (0,3.4641)[circle,fill,inner sep=1.5pt]{};
        \filldraw[black] (-4,-3.5641) node[left] {$g^{\tiny \begin{array}{cc} 1 \\ 2 \end{array}}$};
        \node at (-4,-3.4641)[circle,fill,inner sep=1.5pt]{};
        \filldraw[black] (3.4598,-1.4347) node[right] {$g^3$};
        \node at (3.4598,-1.4347)[circle,fill,inner sep=1.5pt]{};
        \filldraw[black] (-2.1547,-4) node[below] {$g^1$};
        \node at (-1.9547,-4)[circle,fill,inner sep=1.5pt]{};
        \filldraw[black] (-1.4547,2) node[above] {$g^2$};
        \node at (-1.4547,2)[circle,fill,inner sep=1.5pt]{};
        \filldraw[black] (2,-0.2) node[above right] {-$g^{\tiny \begin{array}{cc} 1 \\2 \end{array}}$};
        \node at (2,0)[circle,fill,inner sep=1.5pt]{};
        \filldraw[black] (-1.8,-0.2) node[above left] {-$g^{\tiny \begin{array}{cc} 3 \\1 \end{array}}$};
        \node at (-2,0)[circle,fill,inner sep=1.5pt]{};
        \filldraw[black] (0,-3.46410161514) node[below] {-$g^{\tiny \begin{array}{cc} 2 \\ 3 \end{array}}$};
        \node at (0,-3.46410161514)[circle,fill,inner sep=1.5pt]{};
        \filldraw[black] (2,4.5) node[above] {$\Cfrak_{1}$};
        \filldraw[black] (4.75,-0.5) node[above] {$\Cfrak_{2}$};
        \filldraw[black] (-1.5,-6.25) node[above] {$\Cfrak_{3}$};
        \filldraw[black] (-5,0.5) node[above] {$\Cfrak_{4}$};
        \filldraw[black] (2.8,0.9) node[above] {$\Cfrak_{5}$};
        \filldraw[black] (3,-3.3) node[above] {$\Cfrak_{6}$};
        \filldraw[black] (1,-5) node[above] {$\Cfrak_{7}$};
        \filldraw[black] (-3.1,-3.2) node[above] {$\Cfrak_{8}$};
        \filldraw[black] (-3.65,-1) node[above] {$\Cfrak_{9}$};
        \filldraw[black] (0.3,2.1) node[above] {$\Cfrak_{10}$};
        \filldraw[black] (2.3,-2) node[above] {$\Cfrak_{11}$};
        \filldraw[black] (-1.6,-3) node[above] {$\Cfrak_{12}$};
        \filldraw[black] (0,1.25) node[left] {$\Cfrak_{13}$};
        \filldraw[black] (0,-1.7) node[above] {$\Cfrak_{14}$};
    \end{tikzpicture}
    \caption{Wall-and-chamber structure with chambers corresponding to $\tau$-tilting pairs defined by the surrounding $g$-vectors.} \label{fig:bigwacexmp}
    \end{figure}

    \begin{table}[ht!]
    \centering
    \begin{tabular}{||c | c | c | c | c ||}
     \hline
     Chamber & $\tau$-tilting pair $(T,P)$ & $G_{(T,P)}$ & $C_{(T,P)}$ & torsion class $\Fac T$  \\ [0.5ex] 
     \hline\hline
     $\Cfrak_1$ & $({\tiny\begin{array}{cc} 1 \\ 2 \end{array}} \oplus {\tiny\begin{array}{cc} 2 \\ 3 \end{array}} \oplus {\tiny\begin{array}{cc} 3 \\1 \end{array}}, {\tiny\begin{array}{cc} 0 \end{array}})$ & \tiny$\begin{pmatrix} 1 & 0 & 0 \\ 0 & 1 & 0 \\ 0 & 0 & 1 \end{pmatrix}$ & \tiny$\begin{pmatrix} 1 & 0 & 0 \\ 0 & 1 & 0 \\ 0 & 0 & 1 \end{pmatrix}$ & \small$\mods A$ \\ 
     \hline 
     $\Cfrak_2$ & $({\tiny \begin{array}{c} 3 \end{array}} \oplus {\tiny\begin{array}{cc} 2 \\ 3 \end{array}} \oplus {\tiny\begin{array}{cc} 3 \\1 \end{array}} , {\tiny\begin{array}{cc} 0 \end{array}})$ & \tiny$\begin{pmatrix} -1 & 0 & 0 \\ 0 & 1 & 0 \\ 1 & 0 & 1 \end{pmatrix}$ & \tiny$\begin{pmatrix} -1 & 0 & 0 \\ 0 & 1 & 0 \\ 1 & 0 & 1 \end{pmatrix}$ & \tiny$\add \{{\tiny\begin{array}{cc} 2 \\ 3 \end{array}} \oplus {\tiny\begin{array}{cc} 3 \\1 \end{array}} \oplus {\tiny\begin{array}{cc} 2 \end{array}} \oplus {\tiny \begin{array}{c} 3 \end{array}}\} $ \\
     \hline
     $\Cfrak_3$ & $({\tiny\begin{array}{c} 1 \\ 2 \end{array}} \oplus {\tiny \begin{array}{c} 1 \end{array}} \oplus {\tiny\begin{array}{cc} 3 \\1 \end{array}}, {\tiny\begin{array}{cc} 0 \end{array}})$ & \tiny$\begin{pmatrix} 1 & 1 & 0 \\ 0 & -1 & 0 \\ 0 & 0 & 1 \end{pmatrix}$ & \tiny$\begin{pmatrix} 1 & 1 & 0 \\ 0 & -1 & 0 \\ 0 & 0 & 1 \end{pmatrix}$ & \tiny$\add \{ {\tiny\begin{array}{cc} 1 \\ 2 \end{array}} \oplus {\tiny\begin{array}{cc} 3 \\1 \end{array}} \oplus {\tiny \begin{array}{c} 1 \end{array}} \oplus {\tiny \begin{array}{c} 3 \end{array}} \}$ \\
     \hline
     $\Cfrak_4$ & $({\tiny\begin{array}{cc} 1 \\ 2 \end{array}} \oplus {\tiny\begin{array}{cc} 2 \\ 3 \end{array}} \oplus {\tiny \begin{array}{c} 2 \end{array}}, {\tiny\begin{array}{cc} 0 \end{array}})$ & \tiny$\begin{pmatrix} 1 & 0 & 0 \\ 0 & 1 & 1 \\ 0 & 0 & -1 \end{pmatrix}$ & \tiny$\begin{pmatrix} 1 & 0 & 0 \\ 0 & 1 & 1 \\ 0 & 0 & -1 \end{pmatrix}$ & \tiny$\add \{ {\tiny\begin{array}{cc} 1 \\ 2 \end{array}} \oplus {\tiny\begin{array}{cc} 2 \\ 3 \end{array}}\oplus {\tiny \begin{array}{c} 1 \end{array}} \oplus {\tiny \begin{array}{c} 2 \end{array}} \}$ \\
     \hline
     $\Cfrak_5$ & $({\tiny \begin{array}{c} 3 \end{array}} \oplus {\tiny\begin{array}{cc} 2\\3 \end{array}}, {\tiny\begin{array}{cc} 1 \\ 2 \end{array}})$ & \tiny$\begin{pmatrix} -1 & 0 & -1 \\ 0 & 1 & 0 \\ 1 & 0 & 0 \end{pmatrix}$ & \tiny$\begin{pmatrix} 0 & 0 & 1 \\ 0 & 1 & 0 \\ -1 & 0 & -1 \end{pmatrix}$ & \tiny$\add \{ {\tiny\begin{array}{cc} 2\\3 \end{array}}\oplus {\tiny \begin{array}{c} 2 \end{array}} \oplus {\tiny \begin{array}{c} 3 \end{array}} \}$ \\
     \hline
     $\Cfrak_6$ & $({\tiny \begin{array}{c} 3 \end{array}} \oplus {\tiny\begin{array}{cc} 3 \\1 \end{array}}, {\tiny\begin{array}{cc} 2 \\ 3 \end{array}})$ & \tiny$\begin{pmatrix} -1 & 0 & 0 \\ 0 & 0 & -1 \\ 1 & 1 & 0 \end{pmatrix}$ & \tiny$\begin{pmatrix} -1 & 0 & 0 \\ 1 & 0 & 1 \\ 0 & -1 & 0 \end{pmatrix}$ & \tiny$\add \{{\tiny\begin{array}{cc} 3 \\1 \end{array}}  \oplus {\tiny \begin{array}{c} 3 \end{array}}\}$ \\
     \hline
     $\Cfrak_7$ & $({\tiny \begin{array}{c} 1 \end{array}} \oplus {\tiny\begin{array}{cc} 3 \\ 1 \end{array}}, {\tiny\begin{array}{cc} 2 \\ 3 \end{array}})$ & \tiny$\begin{pmatrix} 1 & 0 & 0 \\ -1 & 0 & -1 \\ 0 & 1 & 0 \end{pmatrix}$ & \tiny$\begin{pmatrix} 1 & 0 & 0 \\ 0 & 0 & 1 \\ -1 & -1 & 0 \end{pmatrix}$ & \tiny $\add \{ {\tiny\begin{array}{cc} 3 \\ 1 \end{array}} \oplus {\tiny \begin{array}{c} 1 \end{array}} \oplus {\tiny \begin{array}{c} 3 \end{array}} \}$ \\
     \hline
     $\Cfrak_8$ & $({\tiny\begin{array}{cc} 1 \\ 2 \end{array}} \oplus {\tiny \begin{array}{c} 1 \end{array}}, {\tiny\begin{array}{cc} 3 \\ 1 \end{array}})$ & \tiny$\begin{pmatrix} 1 & 1 & 0 \\ 0 & -1 & 0 \\ 0 & 0 & -1 \end{pmatrix}$ & \tiny$\begin{pmatrix} 1 & 1 & 0 \\ 0 & -1 & 0 \\ 0 & 0 & -1 \end{pmatrix}$  & \tiny $\add \{ {\tiny\begin{array}{cc} 1 \\ 2 \end{array}} \oplus {\tiny \begin{array}{c} 1 \end{array}} \}$ \\
     \hline
     $\Cfrak_9$ & $({\tiny\begin{array}{cc} 1 \\ 2 \end{array}} \oplus {\tiny \begin{array}{c} 2 \end{array}}, {\tiny\begin{array}{cc} 3 \\ 1 \end{array}})$ & \tiny$\begin{pmatrix} 1 & 0 & 0 \\ 0 & 1 & 0 \\ 0 & -1 & -1 \end{pmatrix}$ & \tiny$\begin{pmatrix} 1 & 0 & 0 \\ 0 & 1 & 0 \\ 0 & -1 & -1 \end{pmatrix}$ & \tiny$\add \{ {\tiny\begin{array}{cc} 1 \\ 2 \end{array}} \oplus {\tiny \begin{array}{c} 1 \end{array}} \oplus {\tiny \begin{array}{c} 2 \end{array}}\}$\\
     \hline
     $\Cfrak_{10}$ & $({\tiny\begin{array}{cc} 2 \\3 \end{array}} \oplus {\tiny \begin{array}{c} 2 \end{array}}, {\tiny\begin{array}{cc} 1 \\ 2 \end{array}})$ & \tiny$\begin{pmatrix} 0 & 0 & -1 \\ 1 & 1 & 0 \\ 0 & -1 & 0 \end{pmatrix}$ & \tiny$\begin{pmatrix} 0 & 1 & 1 \\ 0 & 0 & -1 \\ -1 & 0 & 0 \end{pmatrix}$ & \tiny$\add \{ {\tiny\begin{array}{cc} 2 \\3 \end{array}} \oplus {\tiny \begin{array}{c} 2 \end{array}}\}$\\
     \hline
     $\Cfrak_{11}$ & $({\tiny \begin{array}{c} 3 \end{array}}, {\tiny\begin{array}{cc} 1 \\ 2 \end{array}}  \oplus {\tiny\begin{array}{cc} 2 \\ 3 \end{array}})$ & \tiny$\begin{pmatrix} -1 & -1 & 0 \\ 0 & 0 & -1 \\ 1 & 0 & 0 \end{pmatrix}$ & \tiny$\begin{pmatrix} 0 & 0 & 1 \\ -1 & 0 & -1 \\ 0 & -1 & 0 \end{pmatrix}$ & \tiny $\add \{ {\tiny \begin{array}{c} 3 \end{array}}\}$ \\
     \hline
     $\Cfrak_{12}$ & $({\tiny\begin{array}{cc} 1 \end{array}},{\tiny\begin{array}{cc} 2 \\ 3 \end{array}} \oplus {\tiny\begin{array}{cc} 3 \\1 \end{array}})$ & \tiny$\begin{pmatrix} 1 & 0 & 0 \\ -1 & -1 & 0 \\ 0 & 0 & -1 \end{pmatrix}$ & \tiny$\begin{pmatrix} 1 & 0 & 0 \\ -1 & -1 & 0 \\ 0 & 0 & -1 \end{pmatrix}$ & \tiny $\add\{ {\tiny\begin{array}{cc} 1 \end{array}}\}$\\
     \hline
     $\Cfrak_{13}$ & $({\tiny \begin{array}{c} 2 \end{array}}, {\tiny\begin{array}{cc} 1 \\ 2 \end{array}} \oplus {\tiny\begin{array}{cc} 3 \\ 1 \end{array}})$ & \tiny$\begin{pmatrix} 0 & -1 & 0 \\ 1 & 0 & 0 \\ -1 & 0 & -1 \end{pmatrix}$ & \tiny$\begin{pmatrix} 0 & 1 & 0 \\ -1 & 0 & 0 \\ 0 & -1 & -1 \end{pmatrix}$ & \tiny $ \add\{ {\tiny \begin{array}{c} 2 \end{array}}\}$ \\
     \hline
     $\Cfrak_{14}$ & $({\tiny\begin{array}{cc} 0 \end{array}},{\tiny\begin{array}{cc} 1 \\ 2 \end{array}} \oplus {\tiny\begin{array}{cc} 2 \\ 3 \end{array}} \oplus {\tiny\begin{array}{cc} 3 \\1 \end{array}})$ & \tiny$\begin{pmatrix} -1 & 0 & 0 \\ 0 & -1 & 0 \\ 0 & 0 & -1 \end{pmatrix}$ & \tiny$\begin{pmatrix} -1 & 0 & 0 \\ 0 & -1 & 0 \\ 0 & 0 & -1 \end{pmatrix}$ & \tiny $\add \{ {\tiny\begin{array}{cc} 0 \end{array}} \}$ \\
     \hline
    \end{tabular}
    \caption{Chambers, their $\tau$-tilting pairs, $g$-matrices, $c$-matrices and torsion classes.}
    \label{table:A3table}
    \end{table}


\begin{ack}
MK and HT are grateful to the organisers of the Workshop of the ICRA held in Montevideo, Uruguay, in August 2022, for putting together such a stimulating event.
HT is also grateful to the Scientific Committee of ICRA 2022 for inviting him to give the lecture series whose content is reflected in these notes.
\end{ack}

\begin{funding}
MK is supported by the DFG through the project SFB/TRR 191 Symplectic Structures in Geometry, Algebra and Dynamics (Projektnummer 281071066-TRR 191). HT is supported by the European Union’s Horizon 2020 research and innovation programme under the Marie Sklodowska-Curie grant agreement No 893654. 
\end{funding}


\end{document}